\documentclass[10pt]{amsart}
\usepackage{amssymb}

\newtheorem{lemma}{Lemma}[section]

\newtheorem{theorem}[lemma]{Theorem}
\newtheorem{remark}[lemma]{Remark}
\newtheorem{proposition}[lemma]{Proposition}
\newtheorem{corollary}[lemma]{Corollary}
\newtheorem{definition}[lemma]{Definition}

\begin{document}
\title{Global Euler obstruction, global Brasselet numbers and critical points}
\author{Nicolas Dutertre and Nivaldo G. Grulha Jr.}
\address{Aix-Marseille Universit\'e, CNRS, Centrale Marseille, I2M, UMR 7373,
13453 Marseille, France.}
\email{nicolas.dutertre@univ-amu.fr}

\address{Universidade de S\~{a}o Paulo, Instituto de Ci\^{e}ncias Matem\'{a}ticas e de Computa\c{c}\~{a}o - USP  Av. Trabalhador S\~{a}o-carlense, 400 - Centro, Caixa Postal: 668 - CEP: 13560-970 - S\~{a}o Carlos - SP -
    Brazil.}
\email{njunior@icmc.usp.br}


\begin{abstract}Let $X \subset \Bbb{C}^n$ be an equidimensional complex algebraic set and let $f: X \to \mathbb{C}$ be a polynomial function. For each $c \in \Bbb{C}$, we define the global Brasselet number of $f$ at $c$, a global counterpart of the Brasselet number defined by the authors in a previous work, and the Brasselet number at infinity of $f$ at $c$.
Then we establish several formulas relating these numbers to the topology of $X$ and the critical points of $f$.
\end{abstract}

\maketitle
\markboth{N. Dutertre and N. G. Grulha Jr.}{Global Euler obstruction, global Brasselet numbers and criticical points}

\section{Introduction}

The local Euler obstruction is an invariant defined by MacPherson in \cite{M} as one of the main ingredients in his proof of the Deligne-Grothendieck conjecture on the existence of Chern classes for singular varieties. The local Euler obstruction at $0 \in X$, where $X$ is a sufficietly small representative of the equidimensional analytic germ $(X,0)$, is denoted by ${\rm Eu}_{X}(0)$. After MacPherson's pioneer work, the Euler obstruction was studied by many authors. Let us mention briefly some of the most important results on this subject. If $\mathcal{V}= \{V_i\}_{i=1}^t$ is a Whitney stratification of $X$, then Brylinski, Dubson and Kashiwara \cite{BrylinskiDubsonKashiwara} proved a famous formula that relates the ${\rm Eu}_{\overline{V_i}}$'s to the Euler characteristic of the normal links of the strata. In \cite{LT1}, L\^e and Teissier showed that ${\rm Eu}_X(0)$ is equal to an alterned sum of multiplicities of generic polar varieties of $X$ at $0$. In \cite{BLS}, Brasselet, L\^{e} and Seade proved a Lefschetz type formula for ${\rm Eu}_X(0)$, i.e. they relate ${\rm Eu}_X(0)$ to the topology of the real Milnor fibre on $X$ of a generic linear function. There are also integral formulas for ${\rm Eu}_X(0)$ in \cite{Lo} and \cite{DutertreIsrael}.

 In \cite{BMPS}, Brasselet, Massey, Parameswaran and Seade defined a relative version of the local Euler obstruction, introducing information for a function $f$ defined on the variety $X$, called the Euler obstruction of a function and denoted by ${\rm Eu}_{f,X}(0)$. They prove a Lefschetz type formula for this invariant. The Euler obstruction of a function can be seen as a generalization of the Milnor number (\cite{BMPS,SeadeTibarVerjovsky1,G}). For instance, in \cite{SeadeTibarVerjovsky1}, Seade, Tib\u{a}r and Verjovsky showed that ${\rm Eu}_{f,X}(0)$ is equal up to sign to the number of critical points of a Morsefication of $f$ lying on the regular part of $X$.

In \cite{DutertreGrulhaAdv}, we study topological properties of functions defined on analytic complex varieties. In order to do it, we define an invariant called the Brasselet number, denoted by ${\rm B}_{f,X}(0)$. This number is well defined even when $f$ has arbitrary singularity. When $f$ has isolated singularity we have ${\rm B}_{f,X}(0)= {\rm Eu}_{X}(0)-{\rm Eu}_{f,X}(0)$. We established several formulas for ${\rm B}_{f,X}$, among them a relative version of the multiplicity formula of L\^e and Teissier, a relative version of the Brylinski-Dubson-Kashiwara formula and an integral formula.

In a manner similar to the local case, in \cite{SeadeTibarVerjovsky2}, working with an affine equidimensional singular variety $X \subset \mathbb{C}^{N}$, Seade, Tib\u{a}r and Verjovsky defined the global Euler obstruction, denoted by ${\rm Eu}(X)$. When $X$ is smooth the global Euler obstruction of $X$ coincides with the Euler characteristic of $X$. They prove a global version of the L\^e-Teissier polar multiplicities formula. Later, this formula was generalized
in \cite{ST} to an index formula for MacPherson cycles.

As the Euler obstruction of a function and the Brasselet number are useful to study the singularities of $f$ in the local case, we introduce in this work the global Brasselet numbers and the Brasselet numbers at infinity, in order to  investigate the topological behavior of the singularities, globally and at the infinity, of a given polynomial function $f$ defined on an algebraic variety $X \subset \mathbb{C}^{N}$. The main references  we use in this paper about the study of singularities at infinity  are \cite{Dias,DRT,Tibar}.

In Section 2 we give prerequisities on the topology of complex algebraic sets: stratified Morse functions, the complex link and the normal Morse datum, constructible functions, the local Euler obstruction and the Brasselet number, the global Euler obstruction. In Section 3 we recall the notions of $t$-regularity at infinity and $\rho$-regularity at infinity, some basic results and we adapt them to the stratified setting. 

In Section 4 we define the global Brasselet numbers and the Brasselet numbers at infinity. We compare the global Brasselets number of $f$ with the global Euler obstruction of the fibres of $f$. The relation presented in Corollary \ref{BDKglobal1} can be seen as a global relative version of the local index formula of Brylinsky, Dubson and Kashiwara \cite{BrylinskiDubsonKashiwara}.

In Section 5  we prove several formulas that relate the number of critical points of a Morsefication of a polynomial function $f$ on an algebraic set $X$, to the global Brasselet numbers and the Brasselet numbers at infinity of $f$. The main result in this section is Theorem \ref{TopologyXgenericfibre}. From this result we obtain many interesting corollaires. Corollary \ref{BDKglobal2}, for instance, is a Brylinski--Dubson--Kashiwara type formula for the total Brasselet number at infinity. We also prove a relative version of the polar multiplicity formula of Seade, Tib\u{a}r and Verjovsky (Corollary \ref{BrasseletNumberAndIntersectionMultiplicity2}).

We finish the paper at Section 6, relating the global Euler obstruction of an equidimensional algebraic set $X \subset \mathbb{C}^{n}$ to the Gauss-Bonnet curvature of its regular part and the Gauss-Bonnet curvature of the regular part of its link at infinity. The result is a global counterpart of the formula that the first author established for analytic germs in \cite{DutertreIsrael}.

\section{Prerequisites on the topology of complex algebraic sets}

In this section, we work with a reduced complex algebraic set $X \subset \mathbb{C}^n$ of dimension $d$. We assume that $X$ is equipped with a finite Whitney stratification $\mathcal{V}$ whose strata are connected. We denote by $X_{\rm reg}$ the regular part of $X$, i.e. the union of all the strata of dimension $d$. 

\subsection{Stratified Morse functions}

The main reference for this subject is \cite{GMP}.

\begin{definition}
{\rm Let $x$ be a point in $X$ and let $V_b$ be the stratum that contains it. A degenerate tangent plane of the stratification $\mathcal{V}$ is an element $T$ (of an appropriate Grassmannian) such that $T = \lim_{x_i \to x} T_{x_i} V_a $, where $V_a$ is a stratum that contains $V_b$ in its frontier and where the $x_i$'s belong to $V_a$.}
\end{definition}

\begin{definition}
{\rm A degenerate covector of $\mathcal{V}$ at a point $x \in X$ is a covector which vanishes on a degenerate tangent plane of $\mathcal{V}$ at $x$, i.e., an element $\eta$ of $T^{*}_{x}\mathbb{C}^n$ such that there exists a degenerate tangent plane $T$   of the stratification at $x$ with ${\eta}(T)=0$. }
\end{definition}

Let $f : X \to \mathbb{C}$ be an analytic function. We assume that $f$ is the restriction to $X$ of an analytic function $F : \mathbb{C}^n \to \mathbb{C}$, i.e. $f=F_{\vert X}$. A point $x$ in $X$ is a critical point of $f$ if it is a critical point of $F_{\vert V(x)}$, where $V(x)$ is the stratum containing $x$. 

\begin{definition}
{\rm Let $x$ be a critical point of $f$. We say that $f$ is general at $x$ with respect to the stratification $\mathcal{V}$  if $DF(q)$ is not a degenerate covector of $\mathcal{V}$ at $x$.

We say that $f$ is general with respect to $\mathcal{V}$ if it is general at all critical points with respect to $\mathcal{V}$.}
\end{definition}

\begin{definition}
{\rm Let $x$ be a critical point of $f$. We say that $x$ is a stratified Morse critical point of $f$ if $f$ is general at $x$ and the function $f_{\vert V(x)} : V(x) \to \mathbb{C}$ has a non-degenerate critical point at $x$ when ${\rm dim} V(x) >0$.

We say that that $f$ is a stratified Morse function if it admits only stratified Morse critical points.}  
\end{definition}

\subsection{The complex link and the normal Morse datum}
The complex link is an important object in the study of the topology of complex analytic sets. It is analogous to the Milnor fibre and was studied first in \cite{Le1}. It plays a crucial role in complex stratified Morse theory (see \cite{GMP}) and appears in general bouquet theorems for the Milnor fibre of a function with isolated singularity (see \cite{Le2, Si, Ti}). 

Let $V$ be a stratum of the stratification $\mathcal{V}$ of $X$ and let $x$ be a point in $V$. Let $g : (\mathbb{C}^n,x) \rightarrow (\mathbb{C},0)$ be an analytic complex function-germ such that the differential form $Dg(x)$ is not a degenerate covector of $\mathcal{V}$ at $x$. Let $N^{\mathbb{C}}_{x,V}$ be a normal slice to $V$ at $x$, i.e. $N^{\mathbb{C}}_{x,V}$ is a closed complex submanifold of $\mathbb{C}^n$ which is transversal to $V$ at $x$ and $N^{\mathbb{C}}_{x,V} \cap V =\{x\}$.
\begin{definition}
{\rm The complex link $\mathcal{L}_V^X$ of $V$ is defined by
$$\mathcal{L}_V^X = X\cap N_{x,V}^{\mathbb{C}}  \cap B_{\epsilon}(x)\cap \{g=\delta\} ,$$
where $ 0< \vert \delta \vert \ll \epsilon \ll 1$.  Here $B_{\epsilon}(x)$ is the closed ball of radius $\epsilon$ centered at $x$ .

The normal Morse datum ${\rm NMD}(V)$ of $V$ is the pair of spaces
$${\rm NMD}(V) =\left(X\cap N_{x,V}^{\mathbb{C}} \cap B_{\epsilon}(x), X\cap N_{x,V}^{\mathbb{C}} \cap B_{\epsilon}(x)\cap \{g=\delta\} \right).$$}
\end{definition}
The fact that these two notions are well-defined, i.e. independent of all the choices made to define them, is explained in \cite{GMP}.  

\subsection{Constructible functions}
We start with a presentation of Viro's method of integration with respect to the Euler characteristic with compact support \cite{Vi}. We work in the semi-algebraic setting.

\begin{definition}
{\rm Let $Y \subset \mathbb{R}^n$ be a semi-algebraic set. A constructible function $\alpha : Y \rightarrow \mathbb{Z}$ is a $\mathbb{Z}$-valued function that can be written as a finite sum:
$$\alpha=\sum_{i\in I} m_i {\bf 1}_{Y_i},$$
where $Y_i$ is a semi-algebraic subset of $Y$ and ${\bf 1}_{Y_i}$ is the characteristic function on $Y_i$.}
\end{definition}

The sum and the product of two constructible functions on $Y$ are again constructible. The set of constructible functions on $Y$ is thus a commutative ring, denoted by $F(Y)$.

\begin{definition}
{\rm If $\alpha \in F(Y)$ and $W \subset Y$ is a semi-algebraic set then the Euler characteristic $\chi(W,\alpha)$ is defined by  
$$\chi(W,\alpha)= \sum_{i\in I} m_i \chi_c(W \cap Y_i),$$ where $\alpha=\sum_{i \in I} m_i {\bf 1}_{Y_i}$ and $\chi_c$ is the Euler characterictic of Borel-Moore homology.}
\end{definition}

The Euler characteristic $\chi(W,\alpha)$ is also called the Euler integral of $\alpha$ and 
denoted by $\int_W \alpha d \chi_c$. Here we follow the terminology and notations used in \cite{BMPS,DutertreGrulhaAdv,ST}.

\begin{definition}
{\rm Let $f : Y \rightarrow Z$ be a continuous semi-algebraic map and let $\alpha : Y \rightarrow \mathbb{Z}$ be a constructible function. The pushforward $f_*\alpha$ of $\alpha$ along $f$ is the function $f_* \alpha: Z \rightarrow \mathbb{Z}$ defined by:
$$f_* \alpha(z)=\chi(f^{-1}(z),\alpha) .$$}
\end{definition} 

\begin{proposition}
The pushforward of a  constructible function is a constructible function.
\end{proposition}

\begin{theorem}[Fubini's theorem] Let $f : Y \rightarrow Z$ be a continuous semi-algebraic map and let $\alpha$ be a constructible function on $Y$. Then we have:
$$\chi(Z,f_* \alpha)=\chi(Y,\alpha).$$
\end{theorem}

Let us go back to the complex situation. Here we write $\mathcal{V}=\{V_1,\ldots,V_t\}$ for the Whitney stratification of $X$. 

\begin{definition}\label{ConstructibleAndStratification}
{\rm A constructible function with respect to the stratification $\mathcal{V}$ of $X$ is a function $\alpha: X \to \mathbb{Z}$ which is constant on each stratum $V$ of the stratification.} 
\end{definition}
This means that there exist integers $n_i, i \in \{1,\ldots,t\}$, such that 
$\alpha=\sum_{i=1}^t n_{i} \cdot {\bf 1}_{V_i}$. 
In most of the cases that we will consider, we can use the topological Euler characteristic $\chi$ instead of $\chi_c$. First since each $V_i$ is an even-dimensional submanifold, by Poincar\'e duality $\chi_c(V_i)$ is equal to $\chi(V_i)$ and so 
$\chi(X,\alpha)= \sum_{i=1}^t n_i \chi (V_i)$.
Now let $B \subset \mathbb{C}^n $ be an euclidian closed ball that intersects $X$ transversally (in the stratified sense). We will give four equalities for  $\chi(X \cap B, \alpha)$. By additivity of $\chi_c$, we have 
$$\chi(X \cap B,\alpha) =\chi(X \cap \mathring{B},\alpha) + 
\chi(X \cap \partial B ,\alpha).$$ 
But $X \cap \partial B$ is Whitney stratified by odd dimensional strata and so $\chi(X \cap \partial B)=0$ (see Lemma 5.0.3 in \cite{Schu} or Proposition 1.6 in \cite{McCroryParusinski}). Therefore, we have
$$\chi(X \cap B,\alpha) = \chi(X \cap \mathring{B},\alpha)=
\sum_{i=1}^t n_i \chi_c( V_i \cap \mathring{B}),$$
and by Poincar\'e duality,
$$\chi(X \cap B,\alpha)= \sum_{i=1}^t n_i \chi( V_i \cap \mathring{B}).$$
But each $V_i \cap B$ is a manifold with boundary, so $\chi( V_i \cap \mathring{B})=\chi(V_i \cap B)$ and 
$$\chi(X \cap B,\alpha)= \sum_{i=1}^t n_i \chi( V_i \cap B).$$
Similarly, if $E= \mathbb{C}^n \setminus \mathring{B}$ then
$$\chi(X \cap E,\alpha)= \sum_{i=1}^t n_i \chi_c( V_i \cap E)
=\sum_{i=1}^t n_i \chi_c( V_i \cap \mathring{E})$$ $$=
\sum_{i=1}^t n_i \chi( V_i \cap \mathring{E})=\sum_{i=1}^t n_i \chi( V_i \cap E).$$
If the radius of $B$ is sufficiently big, then $X \cap \partial B$ is homeomorphic to the link at infinity of $X$, denoted by ${\rm Lk}^\infty(X)$, and $X \cap B$ is a retract by deformation of $X$ which implies that $\chi(X)=\chi(X \cap B)$. Since $X \cap B$ is compact, $\chi(X \cap B)=\chi_c (X \cap   B)$ and so, by additivity, $\chi_c (X)= \chi (X)+ \chi_c( X \cap \mathring{E})$. But $X \cap \mathring{E}$ is homeomorphic to the product of ${\rm Lk}^\infty(X)$ and an open interval in $\mathbb{R}$. Since $\chi_c ({\rm Lk}^\infty(X))=0$, by multiplicativity of $\chi_c$ we obtain that $\chi_c(X \cap \mathring{E})=0$ and finally that $\chi(X)=\chi_c(X)$.

\begin{definition}\label{NMI}
{\rm Let $\alpha : X \to Z$ be a constructible function with respect to the stratification $\mathcal{V}$. Its normal Morse index $\eta(V,\alpha)$ along $V$ is defined by
$$\eta(V,\alpha)=\chi({\rm NMD}(V),\alpha )=\chi(X \cap N_{x,V}^{\mathbb{C}} \cap B_{\epsilon}(x),\alpha)-\chi(\mathcal{L}_{V}^X,\alpha),$$
where $x$ is a point in $V$.}
\end{definition}

If $Z\subset X$ is a closed union of strata, then $\eta(V,{\bf 1}_{Z})=1-\chi(\mathcal{L}_{V}\cap Z)$. 

\subsection{The local Euler obstruction and the Brasselet number}

Here we assume that $X$ is equidimensional. 
The Euler obstruction at $x \in X$, denoted by ${\rm Eu}_{X}(x)$, was defined by MacPherson, using $1$-forms and the Nash blow-up (see \cite{M} for the original definition).  An equivalent definition of the Euler obstruction was given by Brasselet and Schwartz in the context of vector fields \cite {BS}. Roughly speaking, ${\rm Eu}_{X}(x)$ is the obstruction for extending a continuous stratified radial vector field around  $x$ in $X$ to a non-zero section of the Nash bundle over the Nash blow-up of $X$.

The Euler obstruction is a constructible function and there are two distinguished bases for the free abelian group of constructible functions: the characteristic functions ${\bf 1}_{\overline{V}}$ and the Euler obstruction ${\rm Eu}_{\overline{V}}$ of the closure $\overline{V}$ of all strata $V$. Moreover, the key role of the Euler
obstruction comes from the following identities (see \cite{ST} p.34 or \cite{Schu} p.292 and p.323-324):
$$\eta(V',{\rm Eu}_{\overline{V}})= 1\hbox{ if }V'=V,$$ and:
$$\eta(V', {\rm Eu}_{\overline{V}})=0 \hbox{ if } V' \neq V.$$

In \cite{BLS}, Brasselet, L\^e and Seade study the Euler obstruction using hyperplane sections, following ideas of Dubson and Kato. Let us assume that $0$ belongs to $X$. 

\begin{theorem}[\cite{BLS}]\label{BLS}
For each generic linear form $l$, there is $\epsilon_0$
such that for any $\epsilon$ with $0<\epsilon<
\epsilon_0$ , the Euler obstruction of $(X,0)$ is equal to:
$${\rm Eu}_X(0)=\chi \big( X \cap B_\epsilon (0) \cap l^{-1}(\delta), {\rm Eu}_X \big),$$ where  $0 < \vert \delta \vert \ll \epsilon \ll 1$.
\end{theorem}

 Let $f : X \rightarrow \mathbb{C}$ be  a holomorphic function. We assume that $f$ has an isolated singularity (or an isolated critical point) at $0$, i.e. that $f$ has no critical point in a punctured neighborhood of $0$ in $X$.

In \cite{BMPS} Brasselet, Massey, Parameswa\-ran and Seade introduced an invariant  which measures, in a way, how far the equality given in Theorem \ref{BLS} is from being true if we replace the generic linear form $l$ with  some other function on $X$ with at most an isolated stratified critical point at $0$. This number is called the Euler obstruction of a function and denoted by $ {\rm Eu}_{f,X}(0)$. 
The following result is the Brasselet, Massey, Parameswaran and Seade formula \cite {BMPS} that compares, in the same point, the local Euler obstruction with the Euler obstruction of a function.

\begin{theorem}\label{BMPS}
Let $ f : X  \to \mathbb C$ be a function with an isolated
singularity at $0$. For $0 < \vert \delta \vert \ll \varepsilon \ll 1$ we have:
$${\rm Eu}_X(0) - {\rm Eu}_{f,X}(0)= \chi \big( X \cap B_\epsilon (0) \cap f^{-1}(\delta), {\rm Eu}_X \big),$$ where  $0 < \vert \delta \vert \ll \epsilon \ll 1$.
\end{theorem}

In \cite{SeadeTibarVerjovsky1}, J. Seade, Tib\u{a}r and Verjovsky show that the Euler obstruction of $f$ is closely
related to the number of Morse points of a Morsefication of $f$, as it is stated in the next proposition.

\begin{proposition}[\cite{SeadeTibarVerjovsky1}]\label{Prop-n-reg} Let $f:X \to  \mathbb{C} $ be
the  an analytic function with isolated singularity at the
origin.
Then: $${\rm Eu}_{f,X}(0)=(-1)^d n_{\rm reg},$$where
$n_{\rm reg}$ is the number of Morse points on $X_{\rm reg}$ in a stratified
Morsefication of $f$ lying in a small neighborhood of $0$.
\end{proposition}

\begin{definition}\label{DefGoodStrat}
{\rm A good stratification of $X$ relative to $f$ is a stratification $\mathcal{V}$ of $X$ which is adapted to $X^f$, (i.e. $X^f$ is a union of strata) , where $X^f= X \cap f^{-1}(0)$, such that $ \{V_{i}\in \mathcal{V}; V_{i} \not\subset X^f\} $ is a Whitney stratification of $X\setminus X^f$ and such that for any pair of strata $(V_{a},V_{b})$ such that $V_{a} \not\subset X^f$ and $V_{b} \subset X^f$, the $(a_f)$-Thom condition is satisfied.}
\end{definition}

Let us now recall the definition of the Brasselet number, defined in \cite{DutertreGrulhaAdv}.
\begin{definition}\label{DefBrassNumb}
{\rm Let $\mathcal{V}$ be a good stratification of $X$ relative to $f$.
We define ${\rm B}_{f,X}(0)$ by:
$${\rm B}_{f,X}(0)= \chi \big(X \cap B_{\epsilon}(0) \cap f^{-1}(\delta),{\rm Eu}_{X} \big),$$ where $0 < \vert \delta \vert \ll \epsilon \ll 1$.}
\end{definition}

\begin{remark}
{\rm Note that if $f$ has a stratified isolated singularity at the origin then, by Theorem \ref{BMPS}, we have that ${\rm B}_{f,X}(0)= {\rm Eu}_{X}(0) - {\rm Eu}_{f,X}(0)$.}
\end{remark}

\subsection{Global Euler obstruction}
Here we assume that $X$ is equidimensional and we write $\mathcal{V}=\{V_1,\ldots,V_t\}$. In \cite{SeadeTibarVerjovsky2}, Seade, Tib\u{a}r and Verjovsky introduced  a global analogous of the Euler obstruction called the global Euler obstruction and denoted by ${\rm Eu}(X)$.  Let $\widetilde{X} \overset{\nu}{\to} X$ denote the Nash modification of $X$, and let us consider a stratified real vector field $v$ on a subset $V \subset X$: this means that the vector field is continuous and tangent to the strata. The restriction of $v$ to $V$ has a well-defined canonical lifting $\widetilde{v}$ to $\nu^{-1}(V)$ as a section of the real bundle underlying the Nash bundle $\widetilde{T} \to \widetilde{X}$.

\begin{definition}\label{GlobalEulerObstruction}
{\rm We say that the stratified vector field $v$ on $X$ is radial-at-infinity if it is defined on the restriction to $X$ of the complement of a sufficiently large ball $B_{M}$ centered at the origin of $\mathbb{C}^{N}$, and 
it is transversal to $S_{R}$, pointing outwards, for any $R > M$. In particular, $v$ is without zeros on $X\setminus B_{M}$.}
\end{definition}

The ``sufficiently large'' radius $M$ is furnished by the following well-known result.

\begin{lemma}
There exists $M \in \mathbb{R}$ such that, for any $R \ge M$, the sphere $S_{R}$ centered at the origin of $\mathbb{C}^{N}$ and of radius $R$ is stratified transversal to $X$, i.e. transversal to all strata of the stratification $\mathcal{V}$.
\end{lemma}

Using this last lemma and inspired by \cite{BS} and \cite{Db1}, Seade, Tib\u{a}r and Verjovsky defined the global Euler obstruction in \cite{SeadeTibarVerjovsky2} as follows:

\begin{definition}
{\rm Let $\tilde{v}$ be the lifting to a section of the Nash bundle $\tilde{T}$ of a radial-at-infinity stratified vector field $v$ over $X\setminus B_{R}$. We call global Euler obstruction of $X$, and denote it by ${\rm Eu}(X)$, the obstruction for extending $\tilde{v}$ as a nowhere zero section of $\widetilde{T}$ within $\nu^{-1}(X\cap B_{R})$.}
\end{definition}

To be precise, the obstruction to extend $\tilde{v}$ as a nowhere zero section of $\widetilde{T}$ within $\nu^{-1}(X\cap B_{R})$ is in fact a relative cohomology class
$$o(\tilde{v}) \in H^{2d}(\nu^{-1}(X\cap B_{R}), \nu^{-1}(X\cap S_{R})) \simeq H_{c}^{2d}(\widetilde{X}).$$

The global Euler obstruction of $X$ is the evaluation of $o(\tilde{v})$ on the fundamental class of the pair $(\nu^{-1}(X\cap B_{R}), \nu^{-1}(X\cap S_{R}))$. Thus ${\rm Eu}(X)$ is an integer and does not depend on the radius of the sphere defining the link at infinity of $X$. Since two radial-at-infinity vector fields are homotopic as stratified vector fields, it does not depend on the choice of $v$ either.

\begin{remark}\label{Remarque}
{\rm The global Euler obstruction has the following properties (see \cite{SeadeTibarVerjovsky2} p. 396):

\begin{enumerate}

\item{if $X$ is non-singular, then ${\rm Eu}(X)=\chi(X)$,}


\item{${\rm Eu}(X)= \chi(X,{\rm Eu}_X)$.}
\end{enumerate}}
\end{remark}

\section{Regularity conditions at infinity}

A natural question is if the  concepts of the Euler obstruction and the Brasselet number of a function could be extended to the global setting, as Seade, Tib\u{a}r and Verjovsky did for the local Euler obstruction, and what kind of information we could obtain with these possible new global invariants.

But, before extending the local notions of the Euler obstruction and Brasselet number of a function, we recall in this section some definitions and results about the study of singularities at infinity and we adapted some results to the stratified setting. The main references for this section are \cite{Dias,DRT,Tibar} and we refer to these papers for details. 

We consider $X \subset \mathbb{C}^{N}$ a reduced algebraic set of dimension $d$. We use coordinates $(x_{1},\dots,x_{N})$ for the space $\mathbb{C}^{N}$ and coordinates $[x_{0}:x_{1}:\dots:x_{N}]$ for the projective space $\mathbb{P}^{N}$. We consider the algebraic closure $\overline{X}$ of $X$ in the complex projective space $\mathbb{P}^{N}$ and we denote by $$H^{\infty}=\big\{[x_{0}:x_{1}:\dots:x_{N}] \mid x_{0}=0\big\},$$  the hyperplane at infinity of the embedding $\mathbb{C}^{N} \subset \mathbb{P}^{N}$.

One may endow $\overline{X}$ with a semi-algebraic Whitney stratification $\mathcal{W}$ such that $X_{\rm reg}$ is a stratum and the part at infinity $\overline{X}\cap H^{\infty}$ is a union of strata. 

Since $\overline{X}$ is projective and since the stratification of $\overline{X}$ is locally finite, it follows that $\mathcal{W}$ has finitely many strata.  We denote by $X_{\rm sing}$ the set of singular points of $X$, i.e. $X_{\rm sing}=X \setminus X_{\rm reg}$.

In order to recall the definition of the $t$-regularity, let us recall first the definition of the conormal spaces.

\begin{definition}\label{Conormal} {\rm  We denote by $C(X)$ the conormal modification of $X$, defined as: 
$$C(X)= {\rm closure}
\big\{(x,H) \in X_{\rm reg} \times \check{\mathbb{P}}^{N-1}\ | \  T_{x}X_{\rm reg} \subset H \big\} \subset \overline{X}\times \check{\mathbb{P}}^{N-1}.$$Let $\pi: C(X) \to \overline{X}$ denote the projection $\pi(x,H)=x$.}
\end{definition}

\begin{definition}\label{RelativeConormal} {\rm Let $g: X \to \mathbb{C}$ be an analytic function defined in some neighbourhood of $X$ in $\mathbb{C}^{N}$. Let $X_{0}$ denote the subset of $X_{\rm reg}$ where $g$ is a submersion. The relative conormal space of $g$ is defined as follows:
$$C_{g}(X)= {\rm closure} \Big \{(x,H) \in X_{0} \times \check{\mathbb{P}}^{N-1} \ | \ T_{x} g^{-1}\left((g(x)\right)  \subset H \big\} \subset \overline{X}\times \check{\mathbb{P}}^{N-1},$$together with the projection $\pi: C_{g}(X) \to \overline{X}$, $\pi(x,H)=x$.}
\end{definition}

Let $f: X \to \mathbb{C}$ be a function  such that $f= F_{\mid X}$, where $F: \mathbb{C}^{N} \to \mathbb{C}$ is a polynomial function. 

Let $\mathbb{X}=\overline{{\rm graph} f}$ be the closure of the graph of $f$ in $\mathbb{P}^{N} \times \mathbb{C}$ and let $\mathbb{X}^{\infty}= \mathbb{X} \cap (H^{\infty}\times \mathbb{C})$. One has the isomorphism ${\rm graph} f \simeq X$.

We consider the affine charts $U_{j} \times \mathbb{C}$ of $\mathbb{P}^{N}\times \mathbb{C}$, where $$U_{j}=\{[x_{0}: \dots : x_{N}] \mid x_{j}\neq 0 \}, j = 0,1,\dots,N.$$ Identifying the chart $U_{0}$ with the affine space $\mathbb{C}^{N}$, we have $\mathbb{X}\cap (U_{0}\times \mathbb{C})= \mathbb{X} \setminus \mathbb{X}^{\infty}= {\rm graph} f$, and $\mathbb{X}^{\infty}$ is covered by the charts $U_{1}\times \mathbb{C},..., U_{N}\times \mathbb{C}$.

If $g$ denotes the projection to the variable $x_{0}$ in some affine chart $U_{j} \times \mathbb{C}$, then the relative conormal $C_{g}(\mathbb{X}\setminus \mathbb{X}^{\infty} \cap U_{j} \times \mathbb{C}) \subset \mathbb{X} \times \check{\mathbb{P}}^{N}$ is well defined. 

With the projection $\pi (y, H) = y$, let us then consider the space $\pi^{-1}(\mathbb{X}^{\infty})$, which is well defined for every chart $U_{j} \times \mathbb{C}$ as a subset of $C_{g}(\mathbb{X}\setminus \mathbb{X}^{\infty} \cap U_{j} \times \mathbb{C}).$

\begin{definition}\label{CovectorsInfinity} {\rm We call space of characteristic covectors at infinity the set $C^{\infty}=\pi^{-1}(\mathbb{X}^{\infty})$. For some $p_{0} \in \mathbb{X}^{\infty}$, we denote $C^{\infty}_{p_{0}}:= \pi^{-1}(p_{0})$.}
\end{definition}
By Lemma 2.8 in \cite{Tibar}, these notions are well-defined, i.e. they do not depend on the chart $U_j$. 

Let $\tau: \mathbb{P}^{N}\times \mathbb{C} \to \mathbb{C}$ denote the second projection. One defines the relative conormal space $C_{\tau}(\mathbb{P}^{N}\times \mathbb{C})$ as in  Definition \ref{RelativeConormal} where the function $g$ is replaced by the mapping $\tau$.

\begin{definition} {\rm 
We say that $f$ is $t$-regular at $p_{0} \in \mathbb{X}^{\infty}$ if $$C_{\tau}(\mathbb{P}^{N}\times \mathbb{C}) \cap C_{p_{0}}^{\infty} = \emptyset.$$
We say that $f^{-1}(t_0)$ is $t$-regular if $f$ is $t$-regular at all points $p_{0}\in \mathbb{X}^{\infty} \cap \tau^{-1}(t_{0}) $. }
\end{definition}


Let us now recall the definition of $\rho$-regularity. 
Let $K \subset \mathbb{C}^{N}$ be some compact (eventually empty) set and let $\rho: \mathbb{C}^{N}\setminus K \to \mathbb{R}_{\geq 0}$ be a proper analytic submersion. 

\begin{definition}[$\rho$-regularity at infinity] {\rm We say that $f$ is $\rho$-regular at $p_{0}\in \mathbb{X}^{\infty}$ if  there is an open neighbourhood
 $U \subset \mathbb{P}^{N}\times \mathbb{C}$ of $p_{0}$ and an open neighbourhood 
 $D \subset \mathbb{C}$ of $\tau(p_{0})$ such that, for all $t \in D$, the fibre $f^{-1}(t) \cap X_{\rm reg} \cap U$ intersects all the levels of the restriction $\rho_{\mid U \cap X_{\rm reg}}$ and this intersection is transversal.

We say that the fibre $f^{-1}(t_{0})$ is $\rho$-regular at infinity if $f$ is $\rho$-regular at all points $p_{0}\in \mathbb{X}^{\infty} \cap \tau^{-1}(t_{0}) $. 
We say that $t_0$ is an asymptotic $\rho$-non-regular value if $f^{-1}(t_0)$ is not $\rho$-regular at infinity. }
\end{definition}

The next proposition relates $t$-regularity to $\rho_{E}$-regularity, where $\rho_E$ denotes the Euclidian norm.

\begin{proposition}
If $f$ is $t$-regular at $p_{0} \in \mathbb{X}^{\infty}$, then $f$ is $\rho_{E}$-regular at $p_{0}$.
\end{proposition}

\begin{proof}
This is just an adaptation to our setting of the proof of Proposition 2.11 in \cite{Tibar}. 
\end{proof}

\begin{corollary}
The set of asymptotic non-$\rho_E$-regular values of $f$ is finite.
\end{corollary}
\proof It is enough to prove that there are only finitely many values $t_0$ such that $f^{-1}(t_0)$ is not $t$-regular. The proof of this fact is as in Corollary 2.12 in \cite{Tibar}. We can equip 
$\mathbb{X}$ with a Whitney stratification such that $\mathbb{X}^\infty$ is a union of strata and such that any pair of strata $(V,W)$, with $V \subset \mathbb{X} \setminus \mathbb{X}^\infty$, $W \subset \mathbb{X}^\infty$ and $W \subset \overline{V} \setminus V$, 
satisfies the Thom $(a_g)$-regularity condition for some function $g$ defining $\mathbb{X}^\infty$ in $\mathbb{X}$. If $\tau^{-1}(t_0)$ is transversal to $\mathbb{X}^\infty$ in the stratified sense, then $f^{-1}(t_0)$ is $t$-regular. But the mapping $\tau_{\vert \mathbb{X}^\infty }: \mathbb{X}^\infty \to \mathbb{C}$ has a finite number of critical values (in the stratified sense).   \endproof

\begin{proposition}
Let   $l: X \to \mathbb{C}$ be a generic linear projection then for all $p_{0} \in \mathbb{X}^{\infty}$, $l$ is $t$-regular at $p_0$.
\end{proposition}

 \begin{proof} 
With $U_{j}$ defined as before, let us work in the chart $U_{j} \times \mathbb{C}$ and with $C_{g}(\mathbb{X}\setminus \mathbb{X}^{\infty} \cap U_{j}\times \mathbb{C})$ with $\pi(x,H)=x$, as defined above. 

Let us suppose that $l$ is not $t$-regular at $p_{0}=(q_{0},t_{0}) \in \mathbb{X}^{\infty}$. It means that there exists a sequence $p_{n}=(q_{n},l(q_{n})) \to p_{0}$, with $q_{n} \in X_{\rm reg}$ such that 
$$T_{p_n} {\rm graph}(l) \cap T_{p_n}g^{-1}(g(p_n)) \subset H_{n},$$  and $ H_{n} \to H$ and a sequence of hyperplanes $\{L_{n}\}$ such that $ \mathbb{C}^{N} \times \{0\} \subset L_{n}$ and   $L_{n} \to H$,  where $(p_0,H) \in \mathcal{C}_{\tau}(\mathbb{P}^{N}\times \mathbb{C})\cap \mathbb{C}_{p_0}^{\infty}$.

Since in fact each $L_{n}= \mathbb{C}^{N}\times \{0\}$, we conclude that $H=\mathbb{C}^{N}\times \{0\}$. Note also that, $$T_{p_n} {\rm graph}(l) \cap T_{p_n}g^{-1}(g(p_n))$$ $$ =
{\rm graph} (l: T_{q_{n}} X_{\rm reg} \to \mathbb{C}) \cap T_{q_{n}} g^{-1}(g(q_{n}))\times \mathbb{C}.$$
This implies that $\lim\limits_{n\to +\infty} l(u_n)=0$ for any bounded sequence $(u_n)$ of vectors such that $u_n \in T_{q_{n}} X_{\rm reg}  \cap T_{q_{n}} g^{-1}(g(q_{n}))$. As in the previous corollary, we can equip $\overline{X}$ with a Whitney stratification such that
$\overline{X} \cap H^\infty$ is a union of strata and such that any pair of strata $(V,W)$, with $V \subset X $, $W \subset \overline{X} \cap H^\infty$ and $W \subset \overline{V} \setminus V$, satisfies the Thom $(a_g)$-regularity. Therefore we see that the axis of the pencil defined by $l$ is not transversal to $\overline{X}\cap H^\infty$. By Lemma 3.1 in \cite{SeadeTibarVerjovsky2}, this is not possible if $l$ is generic.
So we conclude that $l$ generic is $t$-regular.

\end{proof}

We assume now that $X$ is equipped with a finite Whitney stratification $\mathcal{V}= \{V_i\}_{i=1}^t$ such that $V_1,\ldots,V_{t-1}$ are connected, $\overline{V_1},\ldots,\overline{V_t}$ are reduced and $V_t=X_{\rm reg}$. For $i=\{1,\ldots,t\}$, let $f_i : \overline{V_i} \rightarrow \mathbb{C}$ be  the restriction to $\overline{V_i}$ of the polynomial function $F$. Note that $f_t=f$.

\begin{definition}[stratified $t$-regularity at infinity]
{\rm 
We say that $f$ is stratified $t$-regular at $p_{0} \in \mathbb{X}^{\infty}$ if 
for $i=1,\ldots,t$, $f_i$ is $t$-regular at $p_{0}$. 

We say that $f^{-1}(t_0)$ is stratified $t$-regular 
if $f$ is stratified $t$-regular at all points $p_{0}\in \mathbb{X}^{\infty} \cap \tau^{-1}(t_{0}) $. }
\end{definition}

\begin{definition}[stratified $\rho$-regularity at infinity] {\rm 
We say that $f$ is stratified $\rho$-regular at $p_{0}\in \mathbb{X}^{\infty}$ if  if 
for $i=1,\ldots,t$, $f_i$ is $\rho$-regular at $p_{0}$. 

We say that the fibre $f^{-1}(t_{0})$ is stratified $\rho$-regular at infinity if $f$ is stratified $\rho$-regular at all points $p_{0}\in \mathbb{X}^{\infty} \cap \tau^{-1}(t_{0}) $. 
We say that $t_0$ is a stratified asymptotic $\rho$-non-regular value if $f^{-1}(t_0)$ is not stratified $\rho$-regular at infinity. }
 \end{definition}
 
The following statements are easy consequences of the definitions of stratified $t$-regularity and stratified $\rho$-regularity.

\begin{proposition}
Stratified $t$-regularity implies stratified $\rho_E$-regularity.
\end{proposition}

\begin{corollary} 
The set of stratified asymptotic non-$\rho_E$-regular values of $f$ is finite.
\end{corollary}

\begin{corollary} Let $l: X \to \mathbb{C}$ be a generic linear projection, then for all $p_{0} \in \mathbb{X}^{\infty}$, $l$ is stratified $t$-regular (and therefore stratified $\rho_E$-regular) at $p_0$. Moreover the set of stratified asymptotic non-$\rho_E$-regular values of $f$ is empty.
\end{corollary}

\section{Global Brasselet numbers and Brylinsky-Dubson-Kashiwara formulas}

Let $X \subset \mathbb{C}^n$ be  a reduced algebraic set of dimension $d$, equipped with a finite Whitney stratification $\mathcal{V}= \{V_i\}_{i=1}^t$. We assume that $V_1,\ldots,V_{t-1}$ are connected, $\overline{V_1},\ldots,\overline{V_t}$ are reduced and that $V_t=X_{\rm reg}$, where $X_{\rm reg}$ has dimension $d$ . Let $f : X \rightarrow \mathbb{C}$ be  a complex polynomial, restriction to $X$ of a polynomial function $F : \mathbb{C}^n \rightarrow \mathbb{C}$, i.e., $f=F_{\vert X}$. We assume that $f$ has a finite number of critical points $q_1,\ldots,q_s$ and we denote by $\{a_1,\ldots,a_r\}$ the set of stratified asymptotic non-$\rho_E$-regular values of $f$.

For simplicity, we will write $B_R$ for the ball $B_R(0)$ and $S_R$ for $\partial B_R$.

\begin{lemma}
Let $\alpha : X \to \mathbb{Z}$ be a constructible function with respect to $\mathcal{V}$. The function $c \mapsto \chi(f^{-1}(c),\alpha)$ is constant on $\mathbb{C} \setminus 
\Big( \{ f(q_1),\ldots,f(q_s) \} \cup \{a_1,\ldots,a_r \} \Big)$.
\end{lemma}
\proof Let $c \in \mathbb{C} \setminus 
\Big( \{ f(q_1),\ldots,f(q_s) \} \cup \{a_1,\ldots,a_r \} \Big)$ and let us choose $R_c >0$ such that $f^{-1}(c) \cap \{\rho_E \ge R_c\}$ does not contain any critical point of ${\rho_E}_{\vert f^{-1}(c)}$ (here $f^{-1}(c)$ is equipped with the Whitney stratification $\sqcup V_i \cap f^{-1}(c)$). This implies that $f^{-1}(c) \cap B_R$ is a retract by deformation of $f^{-1}(c)$ and that $\chi (f^{-1}(c)) = \chi (f^{-1}(c) \cap B_R)$ for any $R \ge R_c$. Since $f^{-1}(c)$ is stratified $\rho_E$-regular at infinity, there is $R \ge R_c$ and $\epsilon >0$ such that the mapping $$(\rho_E,f) : f^{-1}(D_\epsilon(c)) \cap \{ \rho \ge R \} \to \mathbb{R} \times \mathbb{C},$$ where $D_\epsilon(c)$ is the closed disc of radius $\epsilon$ centered at $c$ in $\mathbb{C}$, is a stratified submersion and so for $c' \in D_\epsilon(c)$, $f^{-1}(c') \cap B_R$ is also a retract by deformation  of $f^{-1}(c')$. But since $c$ is a  regular value of $f$, $$\chi (f^{-1}(c') \cap B_R)= \chi (f^{-1}(c) \cap B_R),$$ for $c'$ in a small neighborhood of $c$. Hence the result is proved for the function $c \mapsto \chi ( f^{-1}(c))$, i.e., when $\alpha = {\bf 1}_X$ (see the discussion after Definition \ref{ConstructibleAndStratification}).

Let $V$ be a stratum of $X$. By additivity, we have 
$$\chi_c( \bar{V} \cap f^{-1}(c)) = \chi_c( V \cap f^{-1}(c)) +
\chi_c(\bar{V} \setminus V \cap f^{-1}(c)),$$
and by the arguments after Definition \ref{ConstructibleAndStratification}, we get that
$$\chi ( \bar{V} \cap f^{-1}(c)) = \chi ( V \cap f^{-1}(c)) +
\chi (\bar{V} \setminus V \cap f^{-1}(c)).$$
But since $\bar{V}$ and $\bar{V} \setminus V$ are algebraic subsets of $X$ stratified by strata of $\mathcal{V}$, $c$ is a regular value and $f^{-1}(c)$ is stratified $\rho_E$-regular at infinity for $f : \bar{V} \to \mathbb{C}$ and $f : \bar{V} \setminus V \to \mathbb{C}$. Therefore the functions 
$$c \mapsto \chi(\bar{V} \cap f^{-1}(c)) \hbox{ and }c \mapsto \chi((\bar{V} \setminus V) \cap f^{-1}(c)),$$ are constant on $\mathbb{C} \setminus 
\Big( \{ f(q_1),\ldots,f(q_s) \} \cup \{a_1,\ldots,a_r \} \Big)$ and so is the function 
$c \mapsto \chi(V \cap f^{-1}(c))$. Hence for $i \in \{1,\ldots,t\}$, the function 
$c \mapsto \chi(V_i \cap f^{-1}(c))$ is constant on $\mathbb{C} \setminus 
\Big( \{ f(q_1),\ldots,f(q_s) \} \cup \{a_1,\ldots,a_r \} \Big)$. This implies that $c \mapsto \chi(f^{-1}(c),\alpha)$ is also constant on $\mathbb{C} \setminus 
\Big( \{ f(q_1),\ldots,f(q_s) \} \cup \{a_1,\ldots,a_r \} \Big)$ for any constructible function $\alpha$. \endproof

\begin{definition}
{\rm When $X$ is equidimensional, we define the global Brasselet number of $f$ at $c$ by
$${\rm B}_{f,c}^X =\chi(f^{-1}(c),{\rm Eu}_X),$$
and the global Euler obstruction of $f$ at $c$ by
$${\rm Eu}_{f,c}^X ={\rm Eu}(X)-{\rm B}_{f,c}^X.$$}
\end{definition}

Let $a \in \mathbb{C}$ and let $R_a >0$ be such that $f^{-1}(a) \cap B_{R_a}$ does not contain any critical point of ${\rho_E}_{\vert f^{-1}(a)}$. Then there exists $\delta_1 >0$ such that  $$f : f^{-1} \Big(D_{\delta_1}(a) \setminus \{a\} \Big) \to D_{\delta_1}(a) \setminus {\{a\}},$$ is a locally trivial topological fibration (this is just a singular version of the Milnor-L\^e fibration) and so $\chi( f^{-1}(c) \cap B_{R_a})$ is constant for $c \in D_{\delta_1}(a) \setminus {\{a\}}$. 

Since $\chi(f^{-1}(c))$ is constant on $\mathbb{C} \setminus 
\Big( \{ f(q_1),\ldots,f(q_s) \} \cup \{a_1,\ldots,a_r \} \Big)$, there exists $\delta_2 >2$ such that $\chi(f^{-1}(c))$ is constant for $c \in D_{\delta_2}(a) \setminus {\{a\}}$. Since 
$$\chi (f^{-1}(c)) = \chi (f^{-1}(c) \cap B_{R_a}) + \chi (f^{-1}(c) \cap \{ \rho_E \ge R_a \}) - \chi (f^{-1}(c) \cap S_{R_a})$$ $$= \chi (f^{-1}(c) \cap B_{R_a}) + \chi (f^{-1}(c) \cap \{ \rho_E \ge R_a \}),$$
we see that $\chi (f^{-1}(c) \cap \{ \rho_E \ge R_a \})$ is constant for $c$ in $D_\delta(a) \setminus \{a\}$, where $\delta={\rm min} \{\delta_1,\delta_2\}$. 

Let $R'_a >0$ be such that $f^{-1}(a) \cap B_{R'_a}$ does not contain any critical point of ${\rho_E}_{\vert f^{-1}(a)}$. Then there exists $\delta'>0$ such that $\chi (f^{-1}(c) \cap \{ \rho_E \ge R'_a \})$ is constant for $c$ in $D_\delta'(a) \setminus \{a\}$. We can suppose that $R'_a > R_a$. Since there are no critical points of ${\rho_E}_{\vert f^{-1}(a)}$ on $\{ R_a \le \rho_E \le R'_a\}$, $a$ is a regular value of $f : \{ R_a \le \rho_E \le R'_a\} \to \mathbb{C}$ (note that a critical point of $f$ on $f^{-1}(a)$ is also a critical point of ${\rho_E}_{\vert f^{-1}(a)}$). Hence there exists $\nu >0$ such that $$f : f^{-1}(D_\nu (a)) \cap \{ R_a \le \rho_E \le R'_a\} \to D_\nu(a),$$ is a locally trivial topological fibration. Let $\nu'= {\rm min} \{\delta,\delta',\nu\}$. For $c$ in $D_{\nu'} (a) \setminus \{a\}$, we have 
$$\chi (f^{-1}(c) \cap \{ \rho_E \ge R_a \})= \chi (f^{-1}(c) \cap \{ \rho_E \ge R'_a \}) + \chi (f^{-1}(c) \cap \{ R_a \le \rho_E \le R'_a \})$$ $$  =\chi (f^{-1}(c) \cap \{ \rho_E \ge R'_a \}) + \chi (f^{-1}(a) \cap \{ R_a \le \rho_E \le R'_a \}).$$
But $\rho : f^{-1}(a) \cap \{ R_a \le \rho_E \le R_{a'} \} \to \mathbb{C}$ is a stratified submersion and so 
$$\chi (f^{-1}(a) \cap \{ R_a \le \rho_E \le R_{a'} \}) = 
\chi( f^{-1}(a) \cap \{  \rho_E = R_a \})=0.$$
We have proved that 
$$\chi (f^{-1}(c) \cap \{ \rho_E \ge R_a \})= \chi (f^{-1}(c) \cap \{ \rho_E \ge R'_a \}),$$
if $c$ is sufficiently close to $a$.
\begin{definition}
{\rm Let $\alpha : X \to \mathbb{Z}$ be a constructible function with respect to $\mathcal{V}$. For any $a \in \mathbb{C}$, we set 
$${\rm B}_{f,a}^{X,\infty}(\alpha) =\lim_{c \to a} \chi(f^{-1}(c) \cap \{ \rho_E \ge R_a\}, \alpha) \hbox{ and } \lambda_{f,a}^{X,\infty} ={\rm B}_{f,a}^{X,\infty}({\bf 1}_X),$$
where $R_a >0$ is such that $f^{-1}(a) \cap \{ \rho_E \ge R_a \} $ does not contain any critical point of ${\rho_E}_{\vert f^{-1}(a) }$.}
\end{definition}
Note that  ${\rm B}_{f,a}^{X,\infty}(\alpha)$ is well-defined since, by the previous considerations, $$\lim_{c \to a} \chi (\overline{V_i} \cap f^{-1}(c) \cap \{\rho_E \ge R_a \}),$$ is well-defined and so is $$\lim_{c \to a} \chi (V_i \cap f^{-1}(c) \cap \{\rho_E \ge R_a \}).$$
\begin{lemma}
Let $\alpha : X \to \mathbb{Z}$ be a constructible function with respect to $\mathcal{V}$. If $c \in \mathbb{C}$ is such that $f^{-1}(c)$ is stratified $\rho_E$-regular at infinity then ${\rm B}_{f,c}^{X,\infty}(\alpha)=0$.
\end{lemma}
\proof It is enough to prove that $\lambda_{f,c}^{X,\infty}=0$.
Since $f^{-1}(c)$ is stratified $\rho_E$-regular at infinity, there is $R_c >0$ and $\epsilon >0$ such that the mapping $$(\rho_E,f) : f^{-1}(D_\epsilon(c)) \cap \{ \rho_E \ge R_c \} \to \mathbb{R} \times \mathbb{C},$$ is a stratified submersion. Let $c' \in f^{-1}(D_\epsilon(c))$. Then the mapping $$\rho_E : \{\rho_E \ge R_c \} \cap f^{-1}(c') \to \mathbb{R},$$ is a proper stratified submersion and so $$\chi (f^{-1}(c') \cap \{\rho_E \ge R_c \}) =
\chi (f^{-1}(c') \cap \{\rho_E = R_c \})=0.$$
\endproof

\begin{definition}
{\rm Let $\alpha : X \to \mathbb{Z}$ be a constructible function with respect to $\mathcal{V}$. We set 
$${\rm B}_{f}^{X,\infty}(\alpha) = \sum_{c \in \mathbb{C}} {\rm B}_{f,c}^{X,\infty}(\alpha)
 \hbox{ and } \lambda_f^{X,\infty} = {\rm B}_{f}^{X,\infty}({{\bf 1}}_X).$$}
\end{definition}

\begin{definition}
{\rm When $X$ is equidimensional, we define the Brasselet numbers at infinity of $f$ by:
$${\rm B}_{f,c}^{X,\infty} ={\rm B}_{f,c}^{X,\infty}({\rm Eu}_X),$$
for $c \in \mathbb{C}$, and the total Brasselet number at infinity of $f$ by:
$${\rm B}_{f}^{X,\infty} ={\rm B}_{f}^{X,\infty}({\rm Eu}_X).$$}
\end{definition}

We start comparing the global Brasselet numbers of $f$ and the Euler obstructions of the fibres of $f$.
\begin{proposition}\label{BrasseletNumberAndEulerObstruction}
Let $a \in \mathbb{C}$, we have
$${\rm B}_{f,a}^X = {\rm Eu}( f^{-1}(a)) + \sum_{j \ \vert \ f(q_j)=a} {\rm Eu}_X (q_j) - {\rm Eu}_{f^{-1}(a)} (q_j).$$
\end{proposition}
\proof By definition,
$${\rm B}_{f,a}^X = \sum_{i=1}^t \chi( V_i \cap f^{-1}(a)) {\rm Eu}_X(V_i).$$
For each $i \in \{1,\ldots,t\}$, let us denote by $\Gamma_i$ the set consisting of the $q_j$'s such that $q_j \in V_i \cap f^{-1}(a)$.  The partition
$$\displaylines{
\qquad  f^{-1}(a) = \left(\sqcup_{i \ \vert \ \Gamma_i = \emptyset} V_i \cap f^{-1}(a) \right) \bigcup \hfill \cr
\hfill \left( \sqcup_{i \ \vert \ \Gamma_i \not= \emptyset} V_i \cap f^{-1}(a) \setminus \Gamma_i \right) 
\bigcup \left( \sqcup_{i \ \vert \ \Gamma_i \not= \emptyset} \Gamma_i \right), \qquad \cr
}$$
gives a Whitney stratification of $f^{-1}(a)$, and 
$$\displaylines{
\quad {\rm Eu}(f^{-1}(a))  = \sum_{i \ \vert \ \Gamma_i = \emptyset} \chi(V_i \cap f^{-1}(a)) {\rm Eu}_{f^{-1}(a)} (V_i \cap f^{-1}(a))  \hfill \cr
\qquad \qquad + \sum_{i \ \vert \ \Gamma_i \not= \emptyset} \chi(V_i \cap f^{-1}(a) \setminus \Gamma_i) {\rm Eu}_{f^{-1}(a)} (V_i \cap f^{-1}(a)\setminus \Gamma_i) \hfill \cr
\hfill + \sum_{i \ \vert \ \Gamma_i \not= \emptyset} \sum_{q \in \Gamma_i} {\rm Eu}_{ f^{-1}(a)} (q). \qquad \cr
}$$
If $\Gamma_i$ is empty then the intersection $V_i \cap f^{-1}(a)$ is transverse (necessarily dim $V_i >0$) and by \cite{Db2}, Proposition IV. 4.1.1 
$${\rm Eu}_X (V_i)= {\rm Eu}_{f^{-1}(a)} (V_i \cap f^{-1}(a)).$$
If $\Gamma_i$ is not empty and ${\rm dim} V_i >0$, then $$\chi(V_i \cap f^{-1}(a)) = \chi(V_i \cap f^{-1}(a) \setminus \Gamma_i) + \# \Gamma_i,$$ and 
${\rm Eu}_X (V_i)= {\rm Eu}_{X \cap f^{-1}(a)} (V_i \cap f^{-1}(a) \setminus \Gamma_i)$ because outside $\Gamma_i$, $f^{-1}(a)$ intersects $V_i$ transversally. 
If $\Gamma_i$ is not empty and ${\rm dim} V_i =0$, then $$\chi(V_i \cap f^{-1}(a) )=1 \hbox{ and }\chi(V_i \cap f^{-1}(a) \setminus \Gamma_i)=0.$$ Therefore we get 
$$\displaylines{
\quad {\rm B}_{f,a}^X = \sum_{i \ \vert \ \Gamma_i = \emptyset} \chi(V_i \cap f^{-1}(a)) {\rm Eu}_{f^{-1}(a)} (V_i \cap f^{-1}(a)) \hfill \cr
\qquad \qquad + \sum_{i \ \vert \ \Gamma_i \not= \emptyset} \chi(V_i \cap f^{-1}(a) \setminus \Gamma_i) {\rm Eu}_{f^{-1}(a)} (V_i \cap f^{-1}(a)\setminus \Gamma_i) \hfill \cr
\qquad \qquad  \qquad \qquad  + \sum_{i \ \vert \ \Gamma_i \not= \emptyset} \sum_{q \in \Gamma_i} {\rm Eu}_X(q) \hfill \cr}$$
$$= {\rm Eu}(f^{-1}(a)) + \sum_{i \ \vert \ \Gamma_i \not= \emptyset} {\rm Eu}_X(q) -{\rm Eu}_{f^{-1}(a)} (q).$$
\endproof
Note that for a regular value $c$ of $f$, ${\rm B}_{f,c}^X ={\rm Eu}(f^{-1}(c))$. Furthermore if $X= \mathbb{C}^n$ then ${\rm Eu}_X (q_j) =1$ and ${\rm Eu}_{f^{-1}(a)} = 1 + (-1)^{n-2}  \mu'(f,q_j)$, where $\mu'(f,q_j)$ is the first Milnor-Teissier number of $f$ at $q_j$, so 
$${\rm B}_{f,a}^X = \chi(f^{-1}(a)) = {\rm Eu} (f^{-1}(a)) + (-1)^{n-1}\sum_{j \ \vert \ f(q_j)=a} \mu' (f,q_j),$$
and we recover Equality (3.3) in \cite{Tibar2004}. 

A direct corollary of the previous proposition is a global relative version of the local index formula of Brylinski, Dubson and Kashiwara.

\begin{corollary}\label{BDKglobal1}
Let $\alpha : X \to \mathbb{Z}$ be a constructible function with respect to $\mathcal{V}=$. For any $a \in \mathbb{C}$, we have 
$$\chi (f^{-1}(a),\alpha) =\sum_{i=1}^t {\rm B}_{f,a}^{\overline{V_i}} \eta (V_i,\alpha).$$
\end{corollary}
\proof We keep the notations of the previous proposition and apply Equality (0.2) of \cite{Tibar2004} to get 
$$\displaylines{
\quad \chi(f^{-1}(a)) = \sum_{i \ \vert \ \Gamma_i = \emptyset} {\rm Eu}(\overline{V_i \cap f^{-1}(a)}) \Big(1-\chi( \mathcal{L}^{f^{-1}(a)}_{V_i \cap f^{-1}(a)}) \Big) \hfill \cr
\qquad \qquad  + \sum_{i \ \vert \ \Gamma_i \not= \emptyset} {\rm Eu}(\overline{V_i \cap f^{-1}(a) \setminus \Gamma_i}) \Big(1-\chi( \mathcal{L}^{f^{-1}(a)}_{V_i \cap f^{-1}(a)\setminus \Gamma_i})\Big) \hfill \cr
\qquad \qquad  \qquad \qquad + \sum_{i \ \vert \ \Gamma_i \not= \emptyset} \sum_{q \in \Gamma_i} 1-\chi(\mathcal{L}_{\{q\}}^{f^{-1}(a)} ) \hfill \cr
}$$
$$\displaylines{
\qquad \qquad \qquad \qquad  = \sum_{i \ \vert \ \Gamma_i \not= V_i} {\rm Eu}(\overline{V_i} \cap f^{-1}(a))  \Big(1-\chi( \mathcal{L}^{f^{-1}(a)}_{V_i \cap f^{-1}(a)\setminus \Gamma_i})\Big) \hfill \cr
\qquad \qquad  \qquad \qquad  \qquad \qquad + \sum_{i \ \vert \ \Gamma_i \not= \emptyset} \sum_{q \in \Gamma_i} 1-\chi(\mathcal{L}_{\{q\}}^{f^{-1}(a)}) , \hfill \cr
}$$
because $\Gamma_i=V_i$ if and only if $V_i$ is just a $0$-dimensional stratum and in this case, $V_i \cap f^{-1}(a) \setminus \Gamma_i = \emptyset$. 
By the previous proposition, we obtain the equality 
$$\displaylines{
\quad \chi(f^{-1}(a)) =  \sum_{i \ \vert \ \Gamma_i \not= V_i}  \hfill \cr
\hfill  \left[ {\rm B}_{f,a}^{\overline{V_i}} -\sum_{j \ \vert \ f(q_j)=a} -{\rm Eu}_{\overline{V_i}} (q_j) + {\rm Eu}_{\overline{V_i} \cap f^{-1}(a)} (q_j) \right]  \left(1-
\chi( \mathcal{L}^{f^{-1}(a)}_{V_i \cap f^{-1}(a)\setminus \Gamma_i})\right) \hfill \cr 
\qquad \qquad + \sum_{i \ \vert \ \Gamma_i \not= \emptyset} \sum_{q \in \Gamma_i} 1-\chi(\mathcal{L}_{\{q\}}^{f^{-1}(a)}), \hfill \cr
}$$
that we  rewrite 
$$\displaylines{
\quad \chi(f^{-1}(a)) = \sum_{i \ \vert \ \Gamma_i \not= V_i} \hfill \cr
\hfill \left[ {\rm B}_{f,a}^{\overline{V_i}} -\sum_{j \ \vert \ f(q_j)=a} -{\rm Eu}_{\overline{V_i}} (q_j) + {\rm Eu}_{\overline{V_i} \cap f^{-1}(a)} (q_j) \right]  \left(1-
\chi( \mathcal{L}^{f^{-1}(a)}_{V_i \cap f^{-1}(a)\setminus \Gamma_i})\right) \hfill \cr
\qquad \qquad + \sum_{j  \ \vert \ f(q_j)=a} 1-\chi(\mathcal{L}_{\{q_j\}}^{f^{-1}(a)}). \hfill \cr
}$$
If $\Gamma_i \not= V_i$ then $f^{-1}(a)$ intersects $V_i \setminus \Gamma_i$ transversally and so 
$$\chi( \mathcal{L}^{f^{-1}(a)}_{V_i \cap f^{-1}(a)\setminus \Gamma_i}) = \chi(\mathcal{L}_{V_i}^X).$$  Hence we have
$$\displaylines{
\ \chi(f^{-1}(a)) = \sum_{i \ \vert \ \Gamma_i \not= V_i} {\rm B}_{f,a}^{\overline{V_i}} \Big(1-\chi(\mathcal{L}_{V_i}^X)\Big) 
+ \sum_{j \ \vert \ f(q_j)=a }  \Big[ 1- \chi(\mathcal{L}_{\{q_j\}}^{f^{-1}(a)} ) \hfill \cr
\hfill + \sum_{i \ \vert \Gamma_i \not= V_i} -{\rm Eu}_{\overline{V_i}} (q_j) \Big(1-\chi(\mathcal{L}_{V_i}^X)\Big) 
+ {\rm Eu}_{\overline{V_i} \cap f^{-1}(a)} (q_j) \left(1-\chi(\mathcal{L}_{V_i \cap f^{-1}(a) \setminus \Gamma_i}^{f^{-1}(a)})\right) \Big]. \quad \cr
}$$
Let us evaluate the second part of this sum and fix $q$ a critical point of $f$ such that $f(q)=a$. Two cases are possible.

If $q$ belongs to a stratum $V_k$ with $\Gamma_k \not= V_k$ then we add the stratum $V_0 =\{q\}$ to the Whitney stratification of $X$. By  the Brylinski, Dubson and Kashiwara index formula (\cite{BrylinskiDubsonKashiwara} or \cite{Schu}, p294), we know that 
$$ 1 = \sum_{i=0}^t {\rm Eu}_{\overline{V_i}} (q) \Big(1-\chi(\mathcal{L}_{V_i}^X)\Big),$$
and so 
$$1= \sum_{i \ \vert \ \Gamma_i \not= V_i} {\rm Eu}_{\overline{V_i}} (q) \Big(1-\chi(\mathcal{L}_{V_i}^X) \Big) + {\rm Eu}_{V_0}(q) \Big(1-\chi(\mathcal{L}_{V_0}^X)\Big),$$
because $q \notin \overline{V_i}$ if $\Gamma_i = V_i$ ($i \ge 1$). But ${\rm Eu}_{V_0}(q)=1$ and $\chi(\mathcal{L}_{V_0}^X)=1$ because a generic linear form is a stratified submersion at $q$ (see \cite{DutertreGrulhaJofSing}, p90 for details). The same index formula applied to $f^{-1}(a)$ gives 
$$ 1= \sum_{i=1}^t {\rm Eu}_{\overline{V_i \cap f^{-1}(a) \setminus \Gamma_i}} (q) \Big(1-\chi( \mathcal{L}^{f^{-1}(a)}_{V_i \cap f^{-1}(a)\setminus \Gamma_i})\Big) + 1- \chi(\mathcal{L}_{\{q\}}^{f^{-1}(a)}).$$
Therefore we get that 
$$\chi(\mathcal{L}_{\{q\}}^{f^{-1}(a)}) = \sum_{i \ \vert \ \Gamma_i \not= V_i} {\rm Eu}_{\overline{V_i} \cap f^{-1}(a)} (q) \Big(1-\chi( \mathcal{L}^{f^{-1}(a)}_{V_i \cap f^{-1}(a)\setminus \Gamma_i})\Big),$$
and so the contribution of $q$ in the second summand of the above sum is zero.  

If $q$ belongs to a stratum $V_k$ with $\Gamma_k=V_k$ then, actually $V_k = \{ q \}$. By the index formula and the same arguments, we find that 
$$1= \sum_{i \ \vert \ \Gamma_i \not= V_i} {\rm Eu}_{\overline{V_i}} (q) \Big(1-\chi(\mathcal{L}_{V_i}^X)\Big) + 1-\chi(\mathcal{L}^X_{\{q\}}),$$
and 
$$1= \sum_{i \ \vert \ \Gamma_i \not= V_i} {\rm Eu}_{\overline{V_i} \cap f^{-1}(a)} (q) \Big(1-\chi( \mathcal{L}^{f^{-1}(a)}_{V_i \cap f^{-1}(a)\setminus \Gamma_i}) \Big) +1-\chi(\mathcal{L}_{\{q\}}^{f^{-1}(a)}).$$
Hence the contribution of $q$ in the above second summand is $1-\chi (\mathcal{L}^X_{\{q\}})$, which we can write ${\rm B}_{f,a}^{\overline{V_k}} \Big(1-\chi (\mathcal{L}_{V_k}^X)\Big)$. Finally we have proved that
$$\chi (f^{-1}(a)) = \sum_{i=1}^t {\rm B}_{f,a}^{\overline{V_i}} \Big(1-\chi (\mathcal{L}_{V_i}^X)\Big),$$
and the theorem follows because both sides of the equality are linear in $\alpha$.
\endproof

\section{Global Brasselet numbers and critical points}
In this section, we prove several formulas that relate the number of critical points of a Morsefication of a polynomial function $f$ on an algebraic set $X$, to the global Brasselet numbers and the Brasselet numbers at infinity of $f$. We note that when $X=\mathbb{C}^n$, similar formulas have already appeared in the literature (\cite{Bro,HaLe,Suzuki,TibarIMRN98,Tibar,SiersmaTibar,Parusinski,ArtalLuengoMelle}).

The setting is the same as in the previous section:
 $X \subset \mathbb{C}^n$ is a reduced algebraic set of dimension $d$, equipped with a finite Whitney stratification $\mathcal{V}= \{V_i\}_{i=1}^t$ such that $V_1,\ldots,V_{t-1}$ are connected, $\overline{V_1},\ldots,\overline{V_t}$ are reduced and $V_t=X_{\rm reg}$ ; $f : X \rightarrow \mathbb{C}$ is  a complex polynomial, restriction to $X$ of a polynomial function $F : \mathbb{C}^n \rightarrow \mathbb{C}$.  We assume that $f$ has a finite number of critical points $q_1,\ldots,q_s$ and we denote by $\{a_1,\ldots,a_r\}$ the set of stratified asymptotic non-$\rho_E$-regular values of $f$.

\begin{definition}\label{morsefication}
{\rm We say that $\tilde{f} : X \rightarrow \mathbb{C}$ is a Morsefication of $f$ if $\tilde{f}$ is a small deformation of $f$ which is a local (stratified) Morsefication at all isolated critical points of $f$.}
\end{definition}

Let $\tilde{f}$ be a Morsefication of $f$. As in the local, we can take $\tilde{f}$ of the form $f+tl$ where $t$ is a sufficiently small complex number and $l$ is the restriction to $X$ of a generic linear form (see Theorem 2.2 in \cite{Le2}). Note that $\tilde{f}$ has two kinds of critical points: those appearing in a small neighborhood of a critical point of $f$ and those appearing at infinity, i.e., outside a ball of sufficiently big radius. We will only consider the first ones.

Let $n_i$, $i=1,\ldots,t$, be the number of critical points of $\tilde{f}$ appearing in a small neighborhood of a critical point of $f$ on the stratum $V_i$. Note that $$n_i \ge \mu^T(f_{\vert V_i}) = \sum_{j \ \vert \ q_j \in V_i} \mu(f_{\vert V_i}, q_j),$$ where $\mu(f_{\vert V_i}, q_j)$ is the Milnor number of $f_{\vert V_i}$ at $q_j$, since we do not assume that $f$ is general with respect to $\mathcal{V}$.

The following theorem relates the number of stratified critical points of $\tilde{f}$ appearing on the stratum $V_i$ to the topology of $X$ and a generic fibre of $f$.
\begin{theorem}\label{TopologyXgenericfibre}
Let $c \in \mathbb{C}$ be a regular value of $f$, which is not a stratified asymptotic non-$\rho_E$-regular value. We have
$$\chi(X)-\chi(f^{-1}(c)) = \sum_{i=1}^t (-1)^{d_i} n_i \Big(1-\chi(\mathcal{L}_{V_i}^X)\Big) -\lambda_f^{X,\infty}.$$
Moreover if $f$ is general with respect to $\mathcal{V}$, then we have
$$\chi(X)-\chi(f^{-1}(c)) = \sum_{i=1}^t (-1)^{d_i} \mu^T(f_{\vert V_i}) \Big(1-\chi(\mathcal{L}_{V_i}^X)\Big) -\lambda_f^{X,\infty}.$$
\end{theorem}
\proof For $x \in X$, let $\varphi(x)=\chi (f^{-1}(\tilde{x}) \cap B_\epsilon(x))=\chi_c (f^{-1}(\tilde{x}) \cap B_\epsilon(x))$, where $\tilde{x}$ is a regular value of $f$ close to $f(x)$ and $0 \le \vert \tilde{x}-f(x) \vert \ll \epsilon \ll 1$. Note that $\varphi(x)=1$ if $x$ is not a critical point of $f$ and so $\varphi$ is a constructible function. By Fubini theorem, we have
$$\chi(X,\varphi)=\chi(\mathbb{C},f_* \varphi),$$
that we can rewrite 
$$\int_X \varphi(x) d \chi_c (x) = \int_{\mathbb{C}} \left[ \int_{f^{-1}(y)} \varphi(x) d\chi_c (x) \right] dy.$$
Let us compute the integral $\int_{f^{-1}(y)} \varphi(x) d\chi_c(x)$ for $y$ in $\mathbb{C}$. Let $\tilde{y}$ be  a regular value of $f$, which is not a stratified asymptotic non-$\rho_E$-regular value. Let $R_y \gg 1$ be such that $f^{-1}(y) \cap B_{R_y}$ is a retract by deformation of $f^{-1}(y)$ and let us denote by $z_1,\ldots,z_l$ the critical points of $f$ in $f^{-1}(y)$. On the one hand, we have
$$\displaylines{
\quad \chi(f^{-1}(\tilde{y}) \cap B_{R_y})= \chi \left( f^{-1}(\tilde{y}) \cap B_{R_y} \setminus (\cup_{i=1}^l \mathring{B}_\epsilon(z_i)) \right) 
 + \sum_{i=1}^l \chi (f^{-1}(\tilde{y}) \cap B_\epsilon(z_i))   \cr
 \hfill =\chi_c \left( f^{-1}(y) \cap B_{R_y} \setminus (\cup_{i=1}^l B_\epsilon(z_i)) \right) + \sum_{i=1}^l  \varphi (z_i) \quad \quad  \quad \cr  
= \chi_c \left( f^{-1}(y) \cap B_{R_y} \setminus \{z_1,\ldots,z_l \} \right) + \sum_{i=1}^l  \varphi (z_i) = \int_{f^{-1}(y) \cap B_{R_y}} \varphi(x) d \chi_c (x). \cr
}$$ 
On the other hand, we have 
$$\displaylines{
\quad \int_{f^{-1}(y) \cap B_{R_y}} \varphi(x) d \chi_c (x)= \int_{f^{-1}(y)} \varphi(x) d \chi_c(x) - \int_{f^{-1}(y) \cap \{ \rho_E > R_y \} }\varphi(x) d \chi_c (x) \quad \quad \cr
\hfill =\int_{f^{-1}(y)} \varphi(x) d \chi_c(x) - \chi_c (f^{-1}(y) \cap \{ \rho_E> R_y \} ). \quad \cr
}$$
But the function $f^{-1}(y) \cap \{ \rho_E > R_y \} \rightarrow ]R_y, + \infty[$, $x \mapsto \rho_E(x) $ is a proper stratified submersion, so
$$\chi_c(f^{-1}(y) \cap \{ \rho_E > R_y \})= \chi_c(f^{-1}(y) \cap S_R) \times \chi_c (]R_y,+\infty [),$$
where $R \in ]R_y,+\infty[$. Since $\chi_c (f^{-1}(y) \cap S_R)=0$, we find that
$$\displaylines{
\quad \int_{f^{-1}(y)} \varphi(x) d \chi_c(x) = \chi (f^{-1}(\tilde{y}) \cap B_{R_y}) \hfill \cr  
\hfill =\chi_c(f^{-1}(\tilde{y}) \cap B_{R_{\tilde{y}}})) - \chi_c (f^{-1}(\tilde{y}) \cap \{ R_y < \rho_E \le R_{\tilde{y}} \}) \quad \quad \cr
\hfill = \chi (f^{-1}(\tilde{y}) \cap B_{R_{\tilde{y}}})) - \chi (f^{-1}(\tilde{y}) \cap \{ R_y \le \rho_E \le R_{\tilde{y}} \}) \quad \quad \cr
\hfill =\chi(f^{-1}(\tilde{y})) -\lambda_{f,y}^{X,\infty}, \quad \cr
}$$
because $\chi (f^{-1}(\tilde{y}) \cap \{ R_y \le \rho_E \le R_{\tilde{y}} \})=
\chi (f^{-1}(\tilde{y}) \cap \{ R_y \le \rho_E \})$.
Therefore, we get
$$\displaylines{
\quad \int_X \varphi(x) d \chi_c (x)= \chi(X)-\sum_{j=1}^s 1-\varphi(q_j) = \int_{\mathbb{C}} \chi(f^{-1}(\tilde{y})) dy - \int_{\mathbb{C}} \lambda_{f,y}^{X,\infty} dy \hfill \cr
\hfill = \chi(f^{-1}(c)) -\sum_{j=1}^r \lambda_{f,a_j}^{X,\infty} = \chi(f^{-1}(c)) - \lambda_f^{X,\infty}, \quad \cr
}$$
and so,
$$\chi(X)-\chi(f^{-1}(c))= \sum_{j=1}^s 1-\chi(f^{-1}(\tilde{q_j}) \cap B_\epsilon(q_j)) - \lambda_f^{X,\infty}.$$
By Theorem 3.2 in \cite{MasseyTopology96} applied to the sheaf $\mathbb{Z}_X^{\bullet}$, we know that
$$1-\chi(f^{-1}(\tilde{q_j}) \cap B_\epsilon(q_j)) = \sum_{i=1}^t (-1)^{d_i} n_{ij} (1-\chi(\mathcal{L}_{V_i}^X)), \eqno(*)$$
where $n_{ij}$ is the number of critical points of a Morsefication of $f$ that lie on $V_i$ in a small neighborhood of $q_j$. Summing over all the critical points of $f$, we obtain the result.
\endproof

\begin{corollary}\label{TopologyXgenericfibreCons}
Let $\alpha : X \rightarrow \mathbb{Z}$ be a constructible function with respect to $\mathcal{V}$ and let $c \in \mathbb{C}$ be  a regular value of $f$, which is not a stratified asymptotic non-$\rho_E$-regular value. We have
$$\chi(X,\alpha) - \chi(f^{-1}(c),\alpha) = \sum_{i=1}^t (-1)^{d_i} n_i \eta(V_i,\alpha) - {\rm B}_f^{X,\infty}(\alpha).$$
Moreover if $f$ is general with respect to $\mathcal{V}$, then we have
$$\chi(X,\alpha)-\chi(f^{-1}(c),\alpha) =\sum_{i=1}^t (-1)^{d_i} \mu^T (f_{\vert V_i}) \eta(V_i,\alpha) - {\rm B}_f^{X,\infty}(\alpha).$$
\end{corollary}
\proof By the previous theorem, the result is true for $\alpha={\bf 1}_{\overline{V_i}}$. Since both sides of the equality are linear in $\alpha$, we see that the result is valid for any constructible function $\alpha$. \endproof
If $X$ is equidimensional then by \cite{SeadeTibarVerjovsky1}, Proposition 2.3, the term $(-1)^{d_i}n_{ij}$ that appears in Equality $(*)$ is equal to ${\rm Eu}_{f,\overline{V_i}}(q_j)$. Hence the above corollary can be refined.
\begin{corollary}\label{TopologyXgenericfibreEqui0}
Assume that $X$ is equidimensional. Let $\alpha : X \rightarrow \mathbb{Z}$ be a constructible function with respect to $\mathcal{V}$ and let $c \in \mathbb{C}$ be  a regular value of $f$, which is not a stratified asymptotic non-$\rho_E$-regular value. We have
$$\chi(X,\alpha) - \chi(f^{-1}(c),\alpha) = \sum_{i=1}^t (-1)^{d_i} n_i \eta(V_i,\alpha) - {\rm B}_f^{X,\infty}(\alpha)=$$
$$ \sum_{i=1}^t \left(\sum_{j=1}^q  {\rm Eu}_{f,\overline{V_i}}(q_j) \right) \eta(V_i,\alpha) - {\rm B}_f^{X,\infty}(\alpha).$$
Moreover if $f$ is general with respect to $\mathcal{V}$, then we have
$$\chi(X,\alpha)-\chi(f^{-1}(c),\alpha) =\sum_{i=1}^t (-1)^{d_i} \mu^T (f_{\vert V_i}) \eta(V_i,\alpha) - {\rm B}_f^{X,\infty}(\alpha).$$
\end{corollary}
An interesting application occurs when $\alpha={\rm Eu}_X$.
\begin{corollary}\label{TopologyXgenericfibreEqui}
Assume that $X$ is equidimensional. Let  $c \in \mathbb{C}$ be  a regular value of $f$, which is not a stratified asymptotic non-$\rho_E$-regular value. We have
$${\rm Eu}_{f,c}^X= {\rm Eu}(X) -{\rm B}_{f,c}^X = (-1)^d n_t -{\rm B}_f^{X,\infty}= \sum_{j=1}^s {\rm Eu}_{f,X}(q_j)-{\rm B}_f^{X,\infty}.$$
Moreover if $f$ is general with respect to $\mathcal{V}$, then we have
$${\rm Eu}_{f,c}^X= {\rm Eu}(X) -{\rm B}_{f,c}^X = \sum_{j \ \vert \ q_j \in X_{\rm reg}} (-1)^d \mu (f_{\vert X_{\rm reg}},q_j)-{\rm B}_f^{X,\infty}.$$
\end{corollary}
\proof By definition, $\chi(X \cap f^{-1}(c),{\rm Eu}_X)={\rm B}_{f,c}^X$ and ${\rm B}_f^{X,\infty}({\rm Eu}_X) ={\rm B}_f^{X,\infty}$. By Remark \ref{Remarque}, $\chi(X,{\rm Eu}_X)={\rm Eu}(X)$. But if $V_i \not= V_t$, then $\eta(V_i,{\rm Eu}_X)=0$ and $\eta(V_t,{\rm Eu}_X)=1$. \endproof

Another corollary is a Brylinski-Dubson-Kashiwara type formula for the global Brasselet number at infinity.
\begin{corollary}\label{BDKglobal2}
Assume that $X$ is equidimensional and let $\alpha : X \to \mathbb{Z}$ be a constructible function with respect to $\mathcal{V}$. We have 
$${\rm B}_f^{X,\infty}(\alpha) =\sum_{i=1}^t {\rm B}_{f}^{\overline{V_i},\infty} \eta (V_i,\alpha).$$
\end{corollary}
\proof Applying Corollaries and \ref{TopologyXgenericfibreEqui0} and \ref{TopologyXgenericfibreEqui} to each set $\overline{V_i}$, we obtain that 
$$\chi(X,\alpha) -\chi(f^{-1}(c),\alpha) + {\rm B}_f^{X,\infty} (\alpha)= $$
$$\sum_{i=1}^t \left[ {\rm Eu}(\overline{V_i}) -{\rm B}_{f,c}^{\overline{V_i}}-{\rm B}_f^{\overline{V_i},\infty} \right] \eta(V_i,\alpha).$$
But we know that $\chi(X,\alpha) = \sum_{i=1}^t  {\rm Eu}(\overline{V_i}) \eta(V_i,\alpha)$ (see \cite{DutertreIsrael}, Corollary 5.4) and that 
$\chi(f^{-1}(c),\alpha) = \sum_{i=1}^t {\rm B}_{f,c}^{\overline{V_i}}\eta(V_i,\alpha).$ \endproof

Let us study what happens if we replace the generic regular value $c$ with any value $a$. First we do not assume that $X$ is equidimensional. For $i=1,\ldots,t$, let $n_i^a$ be the number of critical points of $\tilde{f}$ on $V_i$ appearing in a small neighborhood of a critical point of $f$, but that do not lie in a small neighborhood of a critical point $q$ of $f$ such that $f(q)=a$. Similarly, we set 
$$\mu_a^T(f_{\vert V_i})= \sum_{j \ \vert \ f(q_j) \not= a } \mu(f_{\vert V_j},q_j).$$
\begin{proposition}\label{TopologyXanyfibre}
Let $\alpha : X \rightarrow \mathbb{Z}$ be a constructible function with respect to $\mathcal{V}$ and let $a \in \mathbb{C}$. We have
$$\chi(X,\alpha) - \chi(f^{-1}(a),\alpha) = \sum_{i=1}^t (-1)^{d_i} n_i^a \eta(V_i,\alpha) - {\rm B}_f^{X,\infty}(\alpha) + {\rm B}_{f,a}^{X,\infty}(\alpha).$$
Moreover if $f$ is general with respect to $\mathcal{V}$ then we have
$$\chi(X,\alpha)-\chi(f^{-1}(a),\alpha)= \sum_{i=1}^t (-1)^{d_i} \mu_a^T(f_{\vert V_i}) \eta(V_i,\alpha) - {\rm B}_f^{X,\infty}(\alpha) + {\rm B}_{f,a}^{X,\infty}(\alpha).$$
\end{proposition}
\proof Let $c$ be a generic value (i.e., regular and not stratified asymptotic non-$\rho_E$-regular value) of $f$ close to $a$. Let $R_a \gg 1$ be such that $f^{-1}(a) \cap B_{R_a}$ is a deformation retract of $f^{-1}(a)$. We have
$$\displaylines{
\quad \chi(f^{-1}(c) \cap B_{R_a}) = \chi (f^{-1}(c) \cap B_{R_a} \setminus (\cup_{j \ \vert \ f(q_j)= a} \mathring{B_\epsilon (q_j)}) \hfill \cr
\hfill  + \sum_{j \ \vert \ f(q_j)=a} \chi(f^{-1}(c) \cap B_\epsilon(q_j))\quad \cr}$$  
$$=  \chi_c (f^{-1}(a) \cap B_{R_a} \setminus (\cup_{j \ \vert \ f(q_j)= a} \mathring{B_\epsilon (q_j)}) + \sum_{j \ \vert \ f(q_j)=a} \chi(f^{-1}(c) \cap B_\epsilon(q_j)) $$
$$=  \chi_c (f^{-1}(a) \cap B_{R_a} \setminus (\cup_{j \ \vert \ f(q_j)= a} B_\epsilon (q_j)) + \sum_{j \ \vert \ f(q_j)=a} \chi(f^{-1}(c) \cap B_\epsilon(q_j)) $$
$$ =\chi_c(f^{-1}(a) \cap B_{R_a}) - \sum_{j \ \vert \ f(q_j)=a} 1-\chi(f^{-1}(c) \cap B_\epsilon(q_j)) $$
$$ =\chi(f^{-1}(a) \cap B_{R_a}) - \sum_{j \ \vert \ f(q_j)=a} 1-\chi(f^{-1}(c) \cap B_\epsilon(q_j)) $$
$$= \chi(f^{-1}(a)) - \sum_{j \ \vert \ f(q_j)=a} 1-\chi(f^{-1}(c) \cap B_\epsilon(q_j)).$$
But 
$$\chi(f^{-1}(c)) = \chi (f^{-1}(c) \cap B_{R_a}) + \chi(f^{-1}(c) \cap \{ R_a \le \rho_E \le R_c \}) $$
$$=\chi(f^{-1}(a)) - \sum_{j \ \vert \ f(q_j)=a} 1-\chi(f^{-1}(c) \cap B_\epsilon(q_j)) + \lambda_{f,a}^{X,\infty}.$$
Combining this with the equality proved in Theorem \ref{TopologyXgenericfibre}, we get 
$$\chi(X)- \chi(f^{-1}(a)) = \sum_{j \ \vert \ f(q_j)\not=a} 1-\chi(f^{-1}(c) \cap B_\epsilon(q_j))-\lambda_{f}^{X,\infty} + \lambda_{f,a}^{X,\infty}.$$
Using Massey's results (\cite{MasseyTopology96}, Theorem 3.2), we obtain the result for $\alpha={\bf 1}_X$. The general case easily follows because of the linearity in $\alpha$ of both sides of the equality. 
\endproof
If $X$ is equidimensional, we can refine the above proposition.
\begin{corollary}\label{TopologyXanyfibreEqui}
Let $\alpha : X \rightarrow \mathbb{Z}$ be a constructible function with respect to $\mathcal{V}$ and let $a \in \mathbb{C}$. We have
$$\chi(X,\alpha) - \chi(f^{-1}(a),\alpha) = \sum_{i=1}^t (-1)^{d_i} n_i^a \eta(V_i,\alpha) - {\rm B}_f^{X,\infty}(\alpha) + {\rm B}_{f,a}^{X,\infty}(\alpha)$$
$$= \sum_{i=1}^t \left( \sum_{j \ \vert \ f(q_j) \not= a} {\rm Eu}_{f,\overline{V_i}} (q_j) \right) \eta(V_i,\alpha) - {\rm B}_f^{X,\infty}(\alpha) + {\rm B}_{f,a}^{X,\infty}(\alpha).$$
Moreover if $f$ is general with respect to $\mathcal{V}$ then we have
$$\chi(X,\alpha)-\chi( f^{-1}(a),\alpha)= \sum_{i=1}^t (-1)^{d_i} \mu_a^T(f_{\vert V_i}) \eta(V_i,\alpha) - {\rm B}_f^{X,\infty}(\alpha) + {\rm B}_{f,a}^{X,\infty}(\alpha).$$
\end{corollary}
As above, we can specify these equalities to the case $\alpha={\rm Eu}_X$.
\begin{corollary}\label{TopologyXanyfibreEuler}
Assume that $X$ is equidimensional. Let $a \in \mathbb{C}$. We have
$$\displaylines{ \quad {\rm Eu}_{f,a}^X= {\rm Eu}(X) -{\rm B}_{f,a}^X = (-1)^d n_t^a -{\rm B}_f^{X,\infty}+ {\rm B}_{f,a}^{X,\infty} \hfill \cr
\hfill = \sum_{j \ \vert \ f(q_j) \not= a} {\rm Eu}_{f,X}(q_j)-{\rm B}_f^{X,\infty} + {\rm B}_{f,a}^{X,\infty}. \quad}$$
Moreover if $f$ is general with respect to $\mathcal{V}$, then we have
$${\rm Eu}_{f,a}^X= {\rm Eu}(X) -{\rm B}_{f,a}^X =  \sum_{j \ \vert \ q_j \in X_{\rm reg} \atop f(q_j) \not= a} (-1)^d \mu (f_{\vert X_{\rm reg}},q_j)-{\rm B}_f^{X,\infty} + {\rm B}_{f,a}^{X,\infty}.$$
\end{corollary}
We can also give a version of the Brylinski-Dubson-Kashiwara formula for the Brasselet numbers at infinity ${\rm B}_{f,a}^{X,\infty}$. 
\begin{corollary}\label{BDKglobal3}
Assume that $X$ is equidimensional. Let $\alpha : X \to \mathbb{Z}$ be a constructible function with respect to $\mathcal{V}$ and let  $a \in \mathbb{C}$. We have
$${\rm B}_{f,a}^{X,\infty}(\alpha) =\sum_{i=1}^t {\rm B}_{f,a}^{\overline{V_i},\infty} \eta (V_i,\alpha).$$
\end{corollary}
\proof Apply the previous two corollaries and proceed as in the proof  of Corollary \ref{BDKglobal2}. \endproof

An easy corollary is a relation between ${\rm Eu}_{f,a}^X$ (resp. ${\rm B}_{f,a}^X$) and ${\rm Eu}_{f,c}^X$ (resp. ${\rm B}_{f,c}^X$) where $c$ is generic.
\begin{corollary}\label{Eulerfgenericandany}
Assume that $X$ is equidimensional. Let $c \in \mathbb{C}$ be a regular value of $f$, which is not a stratified asymptotic non-$\rho_E$-regular value, and let $a \in \mathbb{C}$. We have
$$\displaylines{
\quad {\rm Eu}_{f,c}^X-{\rm Eu}_{f,a}^X= {\rm B}_{f,a}^X -{\rm B}_{f,c}^X= (-1)^d (n_t-n_t^a) - {\rm B}_{f,a}^{X,\infty} \hfill  \cr \hfill  = \sum_{j \ \vert \ f(q_j)  = a} {\rm Eu}_{f,X}(q_j)- {\rm B}_{f,a}^{X,\infty}. \quad \cr}$$
Moreover if $f$ is general with respect to $\mathcal{V}$, then we have
$${\rm Eu}_{f,c}^X-{\rm Eu}_{f,a}^X= {\rm B}_{f,a}^X -{\rm B}_{f,c}^X=  \sum_{j \ \vert \ q_j \in X_{\rm reg} \atop f(q_j) = a} (-1)^d \mu (f_{\vert X_{\rm reg}},q_j)- {\rm B}_{f,a}^{X,\infty}.$$
\end{corollary}
When $f$ has no stratified asymptotic non-$\rho_E$-regular values then in all the above equalities, the terms $\lambda_f^{X,\infty}$, $\lambda_{f,a}^{X,\infty}$, ${\rm B}_f^{X,\infty}(\alpha)$, ${\rm B}_{f,a}^{X,\infty}(\alpha)$, ${\rm B}_f^{X,\infty}$ and ${\rm B}_{f,a}^{X,\infty}$ vanish. 
From now on, we assume that $X$ is equidimensional. If $f=l$ is the restriction to $X$ of a generic linear function $L: \mathbb{C}^N \rightarrow \mathbb{C}$, then $l$ has no stratified asymptotic non-$\rho_E$-regular values and moreover $l$ is a stratified Morse function (see \cite{SeadeTibarVerjovsky2}, Lemma 3.1). 

Keeping the notations introduced in \cite{SeadeTibarVerjovsky2}, we denote by $\alpha_X^{(d)}$ the number of (Morse) critical points of $l$ on $X_{\rm reg}$ and by $\alpha_{X,a}^{(d)}$ those not occuring on $l^{-1}(a)$. In this case, if $c$ is a regular value of $l$ then ${\rm Eu}_{l,c}^X =(-1)^d \alpha_X^{(d)}$ and if $a$ is a critical value of $l$, then ${\rm Eu}_{l,a}^X =(-1)^d \alpha_{X,a}^{(d)}$. By the relation between ${\rm B}_{l,a}^X$ and ${\rm Eu}(l^{-1}(a))$, we obtain 
$${\rm Eu} (X) - {\rm Eu}(l^{-1}(a)) = (-1)^d \alpha_{X,a}^{(d)} + \sum_{j \ \vert \ l(q_j)=a} {\rm Eu}_X(q_j) - {\rm Eu}_{l^{-1}(a)} (q_j),$$
where the $q_j$'s are the critical points of $l$.  For a regular value $c$ of $l$, this gives 
$${\rm Eu} (X) - {\rm Eu}(l^{-1}(c)) = (-1)^d \alpha_X^{(d)},$$
and we remark that we have recovered Equality (2), page 401 in \cite{SeadeTibarVerjovsky2}. Based on this equality, Seade, Tib\u{a}r and Verjovsky could express the global Euler obstruction as an alternating sum of global polar invariants. In the sequel, we will establish a relative version of this result for the global Brasselet number and the Brasselet numbers at infinity. 

So we consider a polynomial function $f : X \rightarrow \mathbb{C}$, restriction to $X$ of a polynomial function $F : \mathbb{C}^N \rightarrow \mathbb{C}$. We assume that $f$ has a finite number of critical points $\{q_1,\ldots,q_s\}$. For $a \in \mathbb{C}$, we put $X_a =f^{-1}(a)$. The algebraic set $X_a$ is equidimensional and if $q_1,\ldots,q_u$, $u \le s$, are the critical points of $f$ on $f^{-1}(a)$, then 
$$\mathcal{V}_a = \left(\sqcup_{i=1}^t V_i \cap f^{-1}(a) \setminus \{q_1,\ldots,q_u\} \right) \cup \left( \sqcup_{j=1}^u \{q_j\} \right),$$
is a Whitney stratification of $X_a$.  

Let $L : \mathbb{C}^N \rightarrow \mathbb{C}$ be a linear function and $l : X \rightarrow \mathbb{C}$ be its restriction to $X$. We denote by $\Gamma_{f,l}^X$ the relative polar variety of $f$ and $l$. It is defined as follows:
$$\Gamma_{f,l}^X = \overline{\left\{ x \in X_{\rm reg} \ \vert \ {\rm rank}[df(x),dl(x)] < 2 \right\}}.$$
It is well-known that for $L$ generic, $\Gamma_{f,l}^X$ is a reduced algebraic curve. Moreover if $L$ is generic, we can assume the following fact:
\begin{enumerate}
\item[ ] $l_{\vert X_a} : X_a \rightarrow \mathbb{C}$ is $\rho$-regular at infinity and Morse stratified.
\end{enumerate}
Let $I_X(\Gamma_{f,l}^X,X_a)$ be the global intersection multiplicity of $\Gamma_{f,l}^X$ and $X_a$, namely
$$I_X(\Gamma_{f,l}^X,X_a)= \sum_{p \in \Gamma_{f,l}^X \cap f^{-1}(a)} I_p(\Gamma_{f,l}^X, X_a),$$
where $I_p(\Gamma_{f,l}^X, X_a)$ is the local intersection multiplicity of $\Gamma_{f,l}^X$ and $X_a$ at $p$. If dim$(X)=1$ then $\Gamma_{f,l}^X=X$ and in this case $I_p(\Gamma_{f,l}^X,X_a)$ is the degree of $l:(X,p) \rightarrow (\mathbb{C},a)$, that is the cardinality of $l^{-1}(c) \cap X \cap B_\epsilon(p)$ for $0 < \vert c-a \vert \ll \epsilon \ll 1$.
\begin{proposition}\label{BrasseletNumberAndIntersectionMultiplicity1}
We have
$${\rm B}_{f,a}^X - {\rm B}_{f,a}^{X \cap H} =(-1)^{d-1} I_X (\Gamma_{f,l}^X,X_a)  +\sum_{j=1}^r {\rm Eu}_{f,X}(q_j),$$
where $H$ is a generic hyperplane given by $H=L^{-1}(g)$ for a regular value $g$ of $l_{\vert X_a}$ and $l_{\vert X}$.
\end{proposition}
\proof Let us treat first the case ${\rm dim}(X)>1$. Applying Equality (2) of \cite{SeadeTibarVerjovsky2} that we have mentioned above  to $X_a$ and $X_a \cap L^{-1}(g)$, we get
$${\rm Eu}(X_a) -{\rm Eu}(X_a \cap H) =(-1)^{d-1} \alpha_{X_a}^{(d-1)},$$
where $\alpha_{X_a}^{(d-1)}$ is the number of critical points of $l_{\vert X_a}$ on $X_{\rm reg} \cap (f^{-1}(a) \setminus \{q_1,\ldots,q_u\})$.   
Since $g$ is a regular value of $l_{\vert X_a}$ then 
$${\rm Eu}(X_a \cap H)= \sum_{i=1}^t \chi(V_i \cap f^{-1}(a) \cap H) {\rm Eu}_{X_a \cap H}(V_i \cap H \cap f^{-1}(a)).$$
The hyperplane $H$ intersects $X$ transversally. Furthermore, because the intersections $V_i \cap f^{-1}(a) \cap H$ are transverse, we know that 
$${\rm Eu}_{X_a \cap H}(V_i \cap H \cap f^{-1}(a)) = {\rm Eu}_{X \cap H}(V_i \cap H),$$
which implies that ${\rm Eu}(X_a \cap H)={\rm B}_{f,a}^{X \cap H}$. Applying Proposition \ref{BrasseletNumberAndEulerObstruction}, we obtain
$${\rm B}_{f,a}^X -{\rm B}_{f,a}^{X \cap H}= (-1)^{d-1} \alpha_{X_a}^{(d-1)} + \sum_{j=1}^r {\rm Eu}_X(q_j) -{\rm Eu}_{X_a}(q_j).$$ 
But $I_X (\Gamma_{f,l}^X,X_a)$ is equal to 
$$\alpha_{X_a}^{(d-1)} + \sum_{j=1}^r I_{q_j} (\Gamma_{f,l}^X,X_a).$$
By Corollary 5.2 in \cite{DutertreGrulhaAdv}, we have 
$$(-1)^{d-1} I_{q_j} (\Gamma_{f,l}^X,X_a) ={\rm B}_{f,X}(q_j) -{\rm B}_{f,X \cap H_j} (q_j),$$
where $H_j = L^{-1}(L(q_j))$. Corollary 6.6 in \cite{DutertreGrulhaAdv} implies that 
$$(-1)^{d-1} I_{q_j} (\Gamma_{f,l}^X,X_a) ={\rm B}_{f,X}(q_j) -{\rm Eu}_{X_a}(q_j)= {\rm Eu}_X(q_j) -{\rm Eu}_{f,X}(q_j) -{\rm Eu}_{X_a}(q_j),$$
and so 
$${\rm Eu}_X(q_j)-{\rm Eu}_{X_a}(q_j)=(-1)^{d-1} I_{q_j}(\Gamma_{f,l}^X,X_a) + {\rm Eu}_{f,X}(q_j).$$
If ${\rm dim}(X)=1$ then ${\rm B}_{f,a}^{X \cap H}=0$ and 
$${\rm B}_{f,a}^X = \sum_{p \in f^{-1}(a)} {\rm Eu}_X(p) = \# f^{-1}(a) \setminus \{q_1,\ldots,q_r\} + \sum_{j=1}^r {\rm Eu}_X (q_j).$$
Applying Theorem 3.1 in \cite{BMPS}, it is easy to see that 
$${\rm Eu}_X(q_j)= {\rm Eu}_{f,X}(q_j) + I_{q_j}(\Gamma_{f,l}^X,X_a).$$
Since $I_p(\Gamma_{f,l}^X, X_a)=1$ if $p$ is a regular point of $f$, we obtain the result.
\endproof

By a standard connectivity argument, $I_X(\Gamma_{f,l}^X,X_a)$ does not depend on the choice of the generic linear function $L$. Following Tib\u{a}r's notation \cite{TibarIMRN98}, we denote it by $\gamma_{X,a}^{(d-1)}$. Similarly for $i=2,\ldots,d$, we define 
$$\gamma_{X,a}^{(d-i)} = I_{X \cap H^{i-1}}(\Gamma_{f,l}^{X \cap H^{i-1}}, X_a \cap H^{i-1}),$$
where $H^{i-1}$ is a generic linear space of codimension $i-1$. The following statement is a relative version of the Seade-Tib\u{a}r-Verjovsky polar formula for the global Euler obstruction.
\begin{corollary}\label{BrasseletNumberAndIntersectionMultiplicity2}
We have 
$${\rm B}_{f,a}^X = \sum_{i=1}^d (-1)^{d-i} \gamma_{X,a}^{(d-i)} + \sum_{j=1}^r {\rm Eu}_{f,X}(q_j).$$
\end{corollary}
\proof We apply the previous result to $X \cap H^{i-1}$. Note that for $i \ge 2$, we can choose $H^{i-1}$ generic enough so that $H^{i-1}$ intersects $X_a$ transversally, which implies that $f^{-1}(a)$ intersects $H^{i-1} \cap X$ transversally. \endproof
If $a=c$ is a generic value of $f$, then the above corollary becomes
$${\rm B}_{f,c}^X = \sum_{i=1}^d (-1)^{d-i} \gamma_{X,c}^{(d-i)}.$$
If we apply this to $f=l$, the restriction to $X$ of a generic linear function $L:\mathbb{C}^n \rightarrow \mathbb{C}$, then for $i=1,\ldots,d$, $\gamma_{X,c}^{d-i}$ is exactly equal to the number $\alpha_X^{(d-i)}$ defined in \cite{SeadeTibarVerjovsky2}, which is the number of critical points of a generic linear function on $X_{\rm reg} \cap H^{i-1}$. Combining this fact with the equality
$${\rm B}_{f,c}^X= {\rm Eu}( X) -{\rm Eu}_{f,c}^X={\rm Eu} (X) -(-1)^d \alpha_X^{(d)},$$
we obtain
$${\rm Eu}( X )= \sum_{i=0}^d (-1)^{d-i} \alpha_X^{(d-i)} = \sum_{i=0}^d (-1)^{i} \alpha_X^{(i)},$$
that is, the main result of  \cite{SeadeTibarVerjovsky2}.  \endproof
Another corollary is a characterization of the Brasselet numbers at infinity in terms of critical points of generic linear forms.
\begin{corollary}
Let $a \in \mathbb{C}$ be a stratified asymptotic non-$\rho_E$-regular value of $f$ and let $c \in \mathbb{C}$ be a generic regular value of $f$. We have
$${\rm B}_{f,a}^{X,\infty} = \sum_{i=1}^d (-1)^{d-i} \left( \gamma_{X,c}^{(d-i)} - \gamma_{X,a}^{(d-i)} \right).$$
\end{corollary}
\proof Use the previous corollary and the equality
$${\rm B}_{f,a}^{X,\infty} = {\rm B}_{f,c}^X -{\rm B}_{f,a}^X + \sum_{j=1}^r {\rm Eu}_{f,X}(q_j).$$
\endproof
If $\alpha : X \rightarrow \mathbb{Z}$ is a constructible function relative to $\mathcal{V}$, then the previous equality, combined with the Brylinski-Dubson-Kashiwara formula for $B_{f,a}^{X,\infty}$ proved in Section 3, gives
$${\rm B}_{f,a}^{X,\infty}(\alpha) =\sum_{j=1}^t \left( \sum_{i=1}^{d_j} (-1)^{d_j-i} ( \gamma_{\overline{V_j},c}^{(d_j-i)} - \gamma_{\overline{V_j},a}^{(d_j-i)} ) \right) \eta(V_j,\alpha).$$
In particular for $\alpha={\bf 1}_X$, we get 
$$\lambda_{f,a}^{X,\infty}=\sum_{j=1}^t \left( \sum_{i=1}^{d_j} (-1)^{d_j-i} ( \gamma_{\overline{V_j},c}^{(d_j-i)} - \gamma_{\overline{V_j},a}^{(d_j-i)} ) \right) (1-\chi(\mathcal{L}^X_{V_j})).$$

We end this section with an application. We assume that ${\rm dim}(X) \ge 2$ and that $f$ is general with respect to $\mathcal{V}$. We also suppose that there exists $l: X \rightarrow \mathbb{C}$, restriction to $X$ of a linear form $L: \mathbb{C}^N \rightarrow \mathbb{C}$, such that $l$ has no stratified asymptotic non-$\rho_E$-regular values and is general with respect to $\mathcal{V}$ and such that the mapping $(f,l)_{\vert X_{\rm reg}}: X_{\rm reg} \rightarrow \mathbb{C}^2$ is a submersion. In this situation, Corollary \ref{TopologyXgenericfibreEqui} gives ${\rm Eu}_{f,c}^X=-{\rm B}_f^{X,\infty}$ because $f$ has no critical points on $X_{\rm reg}$. Similarly ${\rm Eu}_{f,c}^{X\cap H} =-B_f^{X \cap H,\infty}$, where $H=L^{-1}(g)$ for a regular value $g$ of $l:X \rightarrow \mathbb{C}$. But, by Proposition \ref{BrasseletNumberAndIntersectionMultiplicity1} applied to $f$ and $l$, we find 
$${\rm B}_{f,c}^X-{\rm B}_{f,c}^{X \cap H}=0={\rm Eu}(X)-{\rm Eu}_{f,c}^X - {\rm Eu}(X \cap H)+{\rm Eu}_{f,c}^{X \cap H}.$$
But ${\rm Eu}(X)={\rm Eu}(X \cap H)$ because $l$ has no critical points on $X_{\rm reg}$. Finally, we obtain that ${\rm B}_f^{X,\infty}={\rm B}_f^{X \cap H,\infty}$. When we apply this equality to $\mathbb{C}^2$, we recover a well-known result.
Indeed, if $X=\mathbb{C}^2$ then $\lambda_f^\infty=0$ because $f_{\vert H}$ is proper and so $\lambda_f^{ H,\infty}=0$. But $\lambda_f^\infty= \sum_{i=1}^r \lambda_{f,a_i}^\infty$. In this case,
$$\lambda_{f,a_i}^\infty = \chi \left(f^{-1}(c) \cap \{ R_a \le \rho_E \le R_{\tilde{a_i}} \} \right),$$
and $f^{-1}(c) \cap \{ R_a \le \rho_E \le R_{\tilde{a_i}} \}$ is a smooth compact orientable surface with at least two boundary components. Therefore $\lambda_{f,a_i}^\infty \le 0$ and so $\lambda_{f,a_i}^\infty=0$ for each $i \in \{1,\ldots,r\}$. Applying Theorem \ref{TopologyXgenericfibre} and Proposition \ref{TopologyXanyfibre}, we find that for all $c \in \mathbb{C}$, $\chi(f^{-1}(c))=1$. Hence, by Suzuki \cite{Suzuki} or H\^a-L\^e's results \cite{HaLe}, $f$ has no bifurcation value at infinity and so $f : \mathbb{C}^2 \rightarrow \mathbb{C}$ is a fibration. \endproof

\section{Global Euler obstruction and the Gauss-Bonnet measure}
In this section, we relate the global Euler obstruction of an equidimensional algebraic set $X \subset \mathbb{C}^N$ to the Gauss-Bonnet curvature of its regular part and the Gauss-Bonnet curvature of the regular part of its link at infinity. Actually the result we will prove is the global counterpart of the formula that the first author established for analytic germs in \cite{DutertreIsrael} and that gave a positive answer to a question of Fu \cite{FuGeoDiff96}. 

Before recalling this formula, let us give a brief presentation of the Lip\-schitz-Killing curvatures.  
In \cite{FuAmerJ94}, Fu developed integral geometry for compact subanalytic sets. Using the technology of the normal cycle, he associated with every compact subanalytic set $X \subset \mathbb{R}^N$ a sequence of curvature measures 
$$\Lambda_0(X,-),\ldots,\Lambda_N(X,-),$$
called the Lipschitz-Killing measures. He proved several integral geometry formulas, among them a Gauss-Bonnet formula and a kinematic formula. Later another description of the measures using stratified Morse theory was given by Broecker and Kuppe \cite{BroeckerKuppe} (see also \cite{BernigBroecker}). The reader can refer to \cite{DutertreGeoDedicata}, Section 2, for a rather complete presentation of these two approaches and for the definition of the Lipschitz-Killing measures. 

Let us give some comments on these Lipschitz-Killing curvatures. If ${\rm dim}(X)=d$ then
$$\Lambda_{d+1}(X,U')=\cdots=\Lambda_N(X,U')=0,$$
for any Borel set $U'$ of $X$ and $\Lambda_d(X,U')= \mathcal{H}_d(U')$, where $\mathcal{H}_d$ is the $d$-dimensional Hausdorff measure in $\mathbb{R}^N$. Furthemore if $X$ is smooth then for any Borel set $U'$ of $X$
and for $k \in \{0,\ldots,d \}$, $\Lambda_k (X,U')$ is related to the classical Lipschitz-Killing-Weil curvature $K_{d-k}$ through the following equality:
$$\Lambda_k (X,U')= \frac{1}{s_{N-k-1}} \int_{U'} K_{d-k}(x) dx,$$
where $s_{N-k-1}$  is the volume of the $N-k-1$-dimensional unit sphere. The measure $\Lambda_0(X,-)$ is also called the Gauss-Bonnet measure. This terminology is justified by the following Gauss-Bonnet formula (see \cite{FuAmerJ94} or \cite{BroeckerKuppe}):  $\Lambda_0(X,X)=\chi(X)$. 

In \cite{DutertreIsrael}, Corollary 6.10, the first author showed that if $(X,0) \subset (\mathbb{C}^N,0)$ is the germ of an equidimensional analytic set then 
$${\rm Eu}_X(0)= \lim_{\epsilon \to 0} \Lambda_0(X \cap B_\epsilon, X_{\rm reg} \cap S_\epsilon).$$
Roughly speaking, this means that the Euler obstruction is in the limit equal to the Gauss-Bonnet curvature of $X \cap B_\epsilon$ within $X_{\rm reg} \cap S_\epsilon$.

The method we will follow is exactly the same as the one used in \cite{DutertreIsrael}, Section 6, and that is why we will omit some details and often refer to \cite{DutertreIsrael}. However, we will see that there are differences between the two cases.

First we will work with closed semi-algebraic sets. So let $X \subset \mathbb{R}^N$ be a closed semi-algebraic set. We assume that $X$ is equipped with a finite Whitney semi-algebraic stratification $X=\cup_{a \in A} V_a$. For $R \gg 1$ sufficiently big, $B_R$ intersects $X$ transversally and so $X \cap B_R$ admits the following Whitney stratification:
$$X \cap B_R = \bigcup_{a \in A} V_a \cap \mathring{B_R} \cup \bigcup_{a \in A} V_a \cap S_R,$$
(note that $V_a \cap S_R = \emptyset$ if $V_a$ is bounded). By the Gauss-Bonnet theorem mentioned above, we have 
$$\chi(X \cap B_R)= \Lambda_0(X \cap B_R, X \cap B_R),$$
which implies that 
$$\chi(X \cap B_R)= \Lambda_0(X \cap B_R, X \cap \mathring{B_R}) + \Lambda_0(X \cap B_R, X \cap S_R).$$
As in \cite{DutertreIsrael}, Section 6, we can deduce that
$$\chi(X \cap B_R) = \Lambda_0(X , X \cap B_R) + \Lambda_0(X \cap B_R, X \cap S_R).$$
Since $\lim_{R \to + \infty}  \chi(X \cap B_R)= \chi(X)$ and since by Corollary 5.7 in \cite{DutertreAdvGeo}, $$\lim_{R \to +\infty} \Lambda_0(X,X \cap B_R),$$ exists and is finite, we find that $$\lim_{R \to +\infty} \Lambda_0(X \cap B_R,X \cap S_R),$$ exists and is finite. 

We can describe this last limit topologically by means of critical points on $X \cap S_R$ of generic linear functions. Namely, as in \cite{DutertreIsrael}, Section 6, we can prove that for almost all $v \in S^{N-1}$, $v^*_{\vert X}$ has a finite number of critical points and there exists $R_v >0$ such that for $R \ge R_v$, the function $v^*_{\vert X \cap S_R}$ is a stratified Morse function. Here $v^* (x)= \langle v, x \rangle$. A critical point $p$ of $v^*_{\vert X \cap S_R}$ is an inwards-pointing critical point of $v^*_{\vert X \cap B_R}$ if 
$$ \nabla (v^*_{\vert V})(p)=\lambda (p) \nabla({\rho_E}_{\vert V})(p),$$
with $\lambda(p)<0$ and where $V$ is the stratum that contains $p$. Let us denote by $\mathcal{I}_{v,R}$ the set of inwards-pointing critical points of $v^*_{\vert X \cap B_R}$.
\begin{proposition}\label{GaussBonnetSemiAlg}
We have 
$$\displaylines{
\qquad \lim_{R \to +\infty } \Lambda_0 ( X \cap B_R, X \cap S_R)  \hfill \cr
\hfill  =\frac{1}{s_{N-1}} \int_{S^{N-1}} \lim_{R \to +\infty} \sum_{p \in \mathcal{I}_{v,R}} (-1)^{\sigma(p)} \cdot {\rm ind}_{nor}(v^*, X \cap S_R, p) dv, \qquad \cr
}$$
where $\sigma(p)$ is the Morse index of $v^*_{\vert V \cap S_R}$, $V$ being the stratum that contains $p$, and ${\rm ind}_{nor}(v^*, X \cap S_R, p)$ is the normal Morse index of $v^*_{\vert X \cap S_R}$ at $p$.
\end{proposition}
\proof See \cite{DutertreIsrael}, Proposition 6.6. \endproof
Let us apply this equality in the case where $X$ is a complex algebraic set. Let $X \subset \mathbb{C}^N$ be an algebraic set. We keep the notation of the previous sections. We consider  a vector $v$ in $S^{2N-1}$ generic as above and we choose $R \ge R_v$.  
Let $p \in V \cap S_R$ be an inwards-pointing critical point of $v^*_{\vert X \cap B_R}$. It is explained in \cite{DutertreIsrael} that in this case ${\rm ind}_{nor}(v^*, X \cap S_R, p)=\eta(V,{\bf 1}_X)$. If we denote by $\{q_1^v,\ldots,q_s^v\}$ the critical points of $v^*_{\vert X}$ then, by stratified Morse theory, we can write
$$\chi(X \cap B_R)= \sum_{j=1}^s {\rm ind}(v^*,X,q_j^v) + \sum_{i=1}^t \eta(V_i,{\bf 1}_X) \sum_{p \in  \mathcal{I}_{v,R}^{V_i} } (-1)^{\sigma(p)},$$
for $R \gg 1$ sufficiently big and where $\mathcal{I}_{v,R}^{V_i} $ is the set of inwards-pointing critical points of $v^*_{\vert X \cap B_R}$ on the stratum $V_i$.  

When we apply this relation to $X=\overline{V}$ where $V$ is a stratum of depth $0$, this gives that $\lim_{R \to + \infty}  \sum_{p \in \mathcal{I}_{v,R}^{V}}  (-1)^{\sigma(p)}$ exists and is finite. Note that if ${\rm dim}(V)=0$ then $\mathcal{I}_{v,R}^V$ is empty and $\lim_{R \to + \infty}  \sum_{ p \in \mathcal{I}_{v,R}^{V}}  (-1)^{\sigma(p)}=0$. Applied to $X=\overline{W}$, for $W$ a stratum of depth 1, it gives that $\lim_{R \to + \infty}  \sum_{p \in  \mathcal{I}_{v,R}^{W}}  (-1)^{\sigma(p)}$ exists and is finite. By induction on the depth of the stratum, we see that for $i\in \{1,\ldots,t\}$, $\lim_{R \to + \infty}  \sum_{p  \in \mathcal{I}_{v,R}^{V_i} } (-1)^{\sigma(p)}$ exists and is finite. Proposition \ref{GaussBonnetSemiAlg} becomes
\begin{proposition}
We have 
$$\lim_{R \to + \infty} \Lambda_0 ( X \cap B_R, X \cap S_R)   =\sum_{i=1}^t \eta( V_i,{\bf 1}_X)  \frac{1}{s_{2N-1}} \int_{S^{2N-1}}  \lim_{R ù\to + \infty} \sum_{p \in \mathcal{I}^{V_i}_{v,R}} (-1)^{\sigma(p) }dv.$$
\end{proposition}

\begin{corollary}
For $i \in \{1,\ldots,t\}$, $\lim_{R \to + \infty} \Lambda_0(\overline{ V_i} \cap B_R,  V_i \cap S_R)$ exists and is finite. Furthermore, we have
$$\lim_{R \to + \infty} \Lambda_0(X\cap B_R, X \cap S_R) = \sum_{i=1}^t \eta( V,{\bf 1}_X) \cdot \lim_{R \to + \infty} \Lambda_0(\overline{ V_i} \cap B_R,  V_i \cap S_R).$$
\end{corollary}
\proof Same proof as Corollary 6.8 in \cite{DutertreIsrael}. \endproof

\begin{theorem}
For any stratum $V_i$, we have 
$$ \lim_{R \to + \infty} \Lambda_0(\overline{ V_i} \cap B_R,  V_i \cap S_R)  =\sum_{e=1}^{ d_{V_i}} \lim_{R \to + \infty} \frac{1}{b_{2e} R^{2e}} \frac{1}{s_{2N-2e-1}} \int_{ V_i \cap B_R} K_{2 d_{V_i}-2e}(x) dx,$$
where $b_{2e}$ is the volume of the $2e$-dimensional unit ball and $d_{V_i}$ is the dimension of $V_i$.
\end{theorem}
\proof Let us treat first the case of a stratum $V$ of depth $0$ i.e., $\overline{V}=V$. This is trivial if ${\rm dim}(V)=0$. If ${\rm dim}(V)>0$, then we have
$$\chi(V \cap B_R)=\Lambda_0(V,V \cap B_R) + \Lambda_0(V \cap B_R, V \cap S_R),$$
and so
$$\lim_{R \to + \infty} \Lambda_0(V\cap B_R,V \cap S_R) = \chi(V) - \lim_{R \to \infty} \Lambda_0(V,V\cap B_R).$$
The result is then just an application of Theorem 4.3 in \cite{DutertreGeoDedicata}. 

If $V$ is a stratum of depth greater or equal to $1$, then we also have
$$\chi(\overline{V} \cap B_R)=\Lambda_0(\overline{V},\overline{V} \cap B_R) + \Lambda_0(\overline{V} \cap B_R, \overline{V} \cap S_R),$$
and 
$$\lim_{R \to + \infty} \Lambda_0(\overline{V} \cap B_R ,\overline{V} \cap S_R) = \chi(V) - \lim_{R \to \infty} \Lambda_0(\overline{V},\overline{V}\cap B_R).$$
By Theorem 3.5 in \cite{DutertreGeoDedicata}, we obtain that
$$\lim_{R \to + \infty} \Lambda_0(\overline{V} \cap B_R ,\overline{V} \cap S_R)=\sum_{k=1}^{d_V} \lim_{R \to + \infty} \frac{\Lambda_k(\overline{V},\overline{V} \cap B_R)}{b_k R^k}.$$   
Using the description of the Lipschitz-Killing measures for complex analytic sets done in \cite{DutertreIsrael}, Section 4, we find that 
$$\displaylines{
 \lim_{R \to + \infty} \Lambda_0(\overline{V} \cap B_R, \overline{V} \cap S_R) = \sum_{e=1}^{d_V} \lim_{R \to +\infty} \frac{1}{b_{2e}R^{2e}} \frac{1}{s_{2N-2e-1}} \int_{V \cap B_R} K_{2d_V-2e}(x) dx \hfill \cr
\hfill + \sum_{W \subset \overline{V} \setminus V} \eta(W,{\bf 1}_V) \sum_{e=1}^{d_W} \lim_{R \to +\infty} \frac{1}{b_{2e}R^{2e}} \frac{1}{s_{2N-2e-1} }\int_{W \cap B_R} K_{2d_W-2e}(x) dx. \quad \cr
}$$
Comparing this equality with the previous corollary and applying the induction hypothesis gives the result. \endproof

\begin{corollary}
If $X$ is equidimensional then
$${\rm Eu}(X)= \lim_{R \to + \infty} \Lambda_0(X \cap B_R, X_{\rm reg} \cap B_R),$$
and 
$$\lim_{R \to + \infty} \Lambda_0(X \cap B_R, X_{\rm reg} \cap S_R)= \sum_{i=0}^{d-1} (-1)^i \alpha_X^{(i)}.$$
\end{corollary}
\proof By Corollary 5.3 in \cite{DutertreIsrael}, we know that
$$ {\rm Eu}(X)= \sum_{e=0}^d \lim_{R \to +\infty} \frac{1}{b_{2e} R^{2e}} \frac{1}{s_{2N-2e-1}} \int_{X_{{\rm reg}} \cap B_R} K_{2(d-e)}(x) dx,$$
and so 
$${\rm Eu}(X) = \lim_{R \to +\infty} \frac{1}{s_{2N-1}} \int_{X_{{\rm reg}} \cap B_R} K_{2d}(x) dx  + \lim_{R \to + \infty} \Lambda_0(X \cap B_R, X_{\rm reg} \cap S_R) $$
$$=  \lim_{R \to + \infty} \Lambda_0(X , X_{\rm reg} \cap B_R)+  \lim_{R \to + \infty} \Lambda_0(X \cap B_R, X_{\rm reg} \cap S_R)$$
$$=  \lim_{R \to + \infty} \Lambda_0(X , X_{\rm reg} \cap \mathring{B_R})+  \lim_{R \to + \infty} \Lambda_0(X \cap B_R, X_{\rm reg} \cap S_R)$$
$$=  \lim_{R \to + \infty} \Lambda_0(X  \cap B_R, X_{\rm reg} \cap \mathring{B_R})+  \lim_{R \to + \infty} \Lambda_0(X \cap B_R, X_{\rm reg} \cap S_R)$$
$$=  \lim_{R \to + \infty} \Lambda_0(X \cap B_R , X_{\rm reg} \cap B_R).$$
To prove the second equality, we use the fact that 
$$\frac{1}{s_{2N-1}} \int_{X_{\rm reg} \cap B_R} K_{2d}(x) dx =  \int_{X_{\rm reg} \cap B_R} {\rm ch}_d (X_{\rm reg}),$$
where ${\rm ch}_d (X_{\rm reg})$ is the $d$-th Chern form on $X_{\rm reg}$ and the exchange formula proved by Shifrin in \cite{Shifrin}, page 103. Passing to the limit as $R \to +\infty$, this gives that 
$$ \lim_{R \to +\infty} \frac{1}{s_{2N-1}} \int_{X_{{\rm reg}} \cap B_R} K_{2d}(x) dx =(-1)^d \alpha_X^{(d)}.$$
We just have to apply Theorem 3.4 in \cite{SeadeTibarVerjovsky2} to conclude. \endproof


\begin{thebibliography}{Referen}

\bibitem{ArtalLuengoMelle} ARTAL, E., LUENGO, I. and MELLE, A.:
On the topology of a generic fibre of a polynomial function,
{\it Comm. Algebra} {\bf 28} (2000), no. 4, 1767--1787. 


\bibitem{BernigBroecker} BERNIG, A., BR\"OCKER, L.: Courbures intrins\`eques dans les cat\'egories analytico-g\'eom\'etriques, {\em Ann. Inst. Fourier (Grenoble)} {\bf 53}(6) (2003), 1897--1924.



\bibitem{BLS} BRASSELET, J.P., L\^{E}, D. T. and SEADE, J.:  Euler obstruction and indices of vector fields, {\it Topology}, {\bf 6} (2000) 1193--1208.

\bibitem{BMPS} BRASSELET, J.P., MASSEY, D. , PARAMESWARAN, A. and SEADE, J.: Euler obstruction and defects of functions on singular varieties, {\it Journal London Math. Soc (2)} {\bf 70} (2004), no. 1, 59--76.

\bibitem{BS} BRASSELET, J.P. and SCHWARTZ, M.H.: Sur les classes de Chern d'un ensemble analytique complexe, {\it Ast\'{e}risque} {\bf 82-83} (1981) 93--147.


\bibitem{BrylinskiDubsonKashiwara} BRYLINSKI, J., DUBSON, A. and  KASHIWARA, M.: Formule de l'indice pour modules holonomes et obstruction d'Euler locale, {\em C. R. Acad. Sci. Paris S\'er. I Math.}  {\bf 293} (1981), 573--576.


\bibitem{BroeckerKuppe}  BR\"OCKER, L. and KUPPE, M.: Integral geometry of tame sets, {\em Geometriae Dedicata} {\bf 82} (2000), 285--323.


\bibitem{Bro} BROUGHTON, S. A.: Milnor number and the topology of polynomial hypersurfaces, {\em Invent. Math.} {\bf 92} (1988) 217--241.


\bibitem{Dias} DIAS, L. R. G.: On regularity conditions at infinity, {\em J. Singul.} {\bf 10} (2014), 54--66.

\bibitem{DRT} DIAS, L. R. G.; RUAS, M. A. S. and  TIB\u{A}R, M.: Regularity at infinity of real mappings and a Morse-Sard theorem, {\em J. Topol.} {\bf 5} (2012), no. 2, 323--340.


\bibitem{Db1} DUBSON, A.: Classes caract\'eristiques des vari\'et\'es singuli\`eres, {\em C. R. Acad. Sci. Paris S\'er. A-B}   {\bf 287}  (1978), no. 4, 237--240.

\bibitem{Db2} DUBSON, A.: Calcul des invariants num\'eriques des singularit\'es et applications, {\em Sonderforschungsbereich 40 Theoretische Mathematik}  Universitaet Bonn (1981).



\bibitem{DutertreAdvGeo}  DUTERTRE, N.:  A Gauss-Bonnet formula for closed semi-algebraic sets, {\em Advances in Geometry} {\bf 8}, no. 1 (2008), 33--51.

\bibitem{DutertreGeoDedicata} DUTERTRE, N.: Euler characteristic and Lipschitz-Killing curvatures of closed semi-algebraic sets, {\em Geom. Dedicata} {\bf 158}, no. 1  (2012),167--189.

\bibitem{DutertreIsrael} DUTERTRE, N.: Euler obstruction and Lipschitz-Killing curvatures, {\em Israel J. Math.} {\bf 213} (2016), 109--137.

\bibitem{DutertreGrulhaAdv} DUTERTRE, N. and GRULHA, JR. N. G.: L\^e-Greuel type formula for the Euler obstruction and applications, {\em Adv. Math.} {\bf 251}  (2014), 127--146.

\bibitem{DutertreGrulhaJofSing} DUTERTRE, N. and GRULHA JR., N. G.:  Some notes on the Euler obstruction of a function, {\em J. Singul.}  {\bf 10} (2014), 82--91.




\bibitem{FuAmerJ94} FU,  J.H.G.: Curvature measures of subanalytic sets,
{\em  Amer. J. Math.}  {\bf 116},  no. 4 (1994), 819--880. 

\bibitem{FuGeoDiff96}   FU,  J.H.G.: Curvature measures and Chern classes of singular varieties,  {\em J. Differential Geom.}  {\bf 39}, no. 2 (1994), 251--280. 


\bibitem{GMP}  GORESKY, M. and MAC-PHERSON, R.: Stratified Morse theory, {\em Springer-Verlag}, Berlin, 1988.


\bibitem{G} GRULHA, JR. N. G.: The Euler obstruction and Bruce-Roberts' Milnor number. {\em Q. J. Math.} 60 (2009), no. 3, 291--302. 


\bibitem{HaLe} H\`A, H.V. and  L\^E, D.T.: Sur la topologie des polyn\^omes complexes, {\em Acta Math. Vietnam.} {\bf 9} (1984), no. 1, 21--32.

\bibitem{Le1} L\^E D.T.: Vanishing cycles on complex analytic sets, {\em Proc. Sympos., Res. Inst. Math. Sci., Kyoto, Univ. Kyoto, 1975}   S\^urikaisekikenky\^usho K\'oky\^uroku, no. 266 (1976), 299--318.


\bibitem{Le2} L\^E D.T.: Complex analytic functions with isolated singularities,  {\em J. Algebraic Geom.}  {\bf 1}  (1992),  no. 1, 83--99.

\bibitem{LT1} L\^{E} D. T. and TEISSIER, B.: Vari\'{e}t\'{e}s polaires locales et classes de Chern des vari\'{e}t\'{e}s singuli\`{e}res, {\em Ann. of Math.} {\bf 114}, (1981), 457--491.


\bibitem{Lo} LOESER, F.:  Formules int\'egrales pour certains invariants locaux des espaces analytiques complexes, {\em Comment. Math. Helv}. {\bf 59} (1984), no. 2, 204-225.

\bibitem{M} MAC-PHERSON, R. D.: Chern classes for singular algebraic varieties, {\em Ann. of Math.} {\bf 100} (1974), 423--432.

\bibitem{McCroryParusinski} McCRORY, C. and PARUSINSKI, A.: Algebraically constructible functions, {\em Ann. Sci. \'Ecole Norm. Sup. (4)} {\bf 30} (1997), no. 4, 527-552. 


\bibitem{MasseyTopology96} MASSEY, D. B.: Hypercohomology of Milnor fibres, {\em Topology} {\bf 35} (1996), no. 4, 969--1003.

\bibitem{Parusinski} PARUSI\'NSKI, A.: On the bifurcation set of complex polynomial with isolated singularities at infinity, {\em Compositio Math.} {\bf 97} (1995), no. 3, 369--384. 


\bibitem{Schu} SCH\"URMANN, J.: Topology of singular spaces and constructible sheaves, {\em Instytut Matematyczny Polskiej Akademii Nauk. Monografie Matematyczne (New Series)}, {\bf 63} Birkhauser Verlag, Basel, 2003.

\bibitem{ST} SCH\"URMANN, J. and  TIB\u{A}R, M.: Index formula for MacPherson cycles of affine algebraic varieties, {\em Tohoku Math. J. (2)} {\bf 62} (2010), no. 1, 29--44.


\bibitem{SeadeTibarVerjovsky1} SEADE, J., TIB\u{A}R, M. and VERJOVSKY A.: Milnor Numbers and Euler obstruction, {\em Bull. Braz. Math. Soc. (N.S)} {\bf 36} (2005), no. 2,  275--283.

\bibitem{SeadeTibarVerjovsky2} SEADE, J., TIB\u{A}R, M. and VERJOVSKY, A. : Global Euler obstruction and polar invariants, {\em Math. Ann.}  {\bf 333}  (2005),  no. 2, 393--403.

\bibitem{Si} SIERSMA, D.: A bouquet theorem for the Milnor fibre,  {\em J. Algebraic Geom.}  {\bf 4}  (1995),  no. 1, 51--66.

\bibitem{SiersmaTibar} SIERSMA, D.and TIB\u{A}R, M.: Singularities at infinity and their vanishing cycles, {\em  Duke Math. J.} {\bf 80} (1995), no. 3, 771--783.


\bibitem{Shifrin}  SHIFRIN, T.: Formules cin\'ematiques dans le cas complexe, {\em Introduction \`a la th\'eorie des singularit\'es, II}, 97--108, Travaux en Cours, {\bf 37}, Hermann, Paris, 1988.
   
\bibitem{Suzuki} SUZUKI, M.:  Propri\'et\'es topologiques des polyn\^omes de deux variables complexes, et automorphismes alg\'ebriques de l'espace $\mathbb{C}^2$, {\em J. Math. Soc. Japan} {\bf 26} (1974), 241--257. 


\bibitem{Ti} TIB\u{A}R, M.: Bouquet decomposition of the Milnor fibre,  {\em Topology}  {\bf 35}  (1996),  no. 1, 227--241.

\bibitem{TibarIMRN98} TIB\u{A}R, M.:   Asymptotic equisingularity and topology of complex hypersurfaces, {\em Internat. Math. Res. Notices} (1998), no. 18, 979--990. 


\bibitem{Tibar} TIB\u{A}R, M.:   Regularity at infinity of real and complex polynomial functions,
{\em Singularity theory (Liverpool, 1996)}, London Math. Soc. Lecture Note Ser., {\bf 263}, 249--264, Cambridge Univ. Press, Cambridge, 1999. 

\bibitem{Tibar2004} TIB\u{A}R, M.: Duality of Euler data for affine varieties,  Singularities in geometry and topology 2004, 251--257,
{\em Adv. Stud. Pure Math.}, {\bf 46}, Math. Soc. Japan, Tokyo, 2007. 

\bibitem{Vi} VIRO, O.: Some integral calculus based on Euler characteristic,
{\em Topology and geometry, Rohlin Seminar},  127--138, Lecture Notes in Math., {\bf 1346}, Springer, Berlin, 1988.


\end{thebibliography}
\end{document}